\documentclass{amsart}

\newtheorem{teo}{Theorem}[section] 
 
\newtheorem{lemma}[teo]{Lemma} 
 
\newtheorem{prop}[teo]{Proposition}
\newtheorem{corollario}[teo]{Corollary}
\theoremstyle{definition} 
\newtheorem{definiz}[teo]{Definition} 

\newenvironment{proof1}{\noindent\textit{Proof.}}{\hfill$\square$\medskip\\}
\newenvironment{proof2}{\noindent\textit{Proof of proposition
\ref{distanza}. }}{\hfill$\square$\medskip\\}

\newcommand{\ratio}{{\mathcal{CR}}}
\newcommand{\GL}{{\mathbb{M}(\mathbb{H}^{+})}}
\newcommand{\GLH}{{SL(\mathbb{H}^+)}}
\newcommand{\hh}{{\mathbb{H}}}
\newcommand{\rr}{{\mathbb{R}}}
\newcommand{\s}{{\mathbb{S}}}

\usepackage{epsfig,amsfonts,amssymb}
\usepackage{color}

\theoremstyle{remark} 
\newtheorem{nota}[teo]{Remark}
\numberwithin{equation}{section}

\begin{document}
\title[Poincar\'{e} metric on $\Delta_{\mathbb{H}}$]{
M\"obius transformations and the Poincar\'e distance in the
quaternionic setting}
 \author[C. Bisi, G. Gentili]
{Cinzia Bisi $^\ast$, Graziano Gentili $^\ast$}
\date{May 28th, 2008}
\thanks{\rm $^\ast$ Partially supported by Progetto MIUR di
Rilevante Interesse Nazionale {\it Propriet{\`a} geometriche
delle variet{\`a} reali e complesse} and by GNSAGA - INDAM}
\address{First author: Dipartimento di Matematica, Universit\'a della
Calabria, Ponte Bucci, Cubo 30b, Arcavacata di Rende (CS), 87036,
Italy}
\address{Second author: Dipartimento di Matematica, Universit\'a di
Firenze, Viale Morgagni 67/A, 50134 Firenze, Italy}
\vspace{1cm}
\email{bisi@math.unifi.it}
\email{bisi@mat.unical.it}
\email{gentili@math.unifi.it}

\subjclass{Primary: 30G35 Secondary: 30C20, 30F45}
\keywords{Functions of hypercomplex variables; quaternionic,
M\"obius transformations; quaternionic Poincar\'e distance and metric}

\begin{abstract} In the space $\hh$ of quaternions, we investigate the natural, invariant 
    geometry of the open, unit disc $\Delta_{\hh}$ and of the open
    half-space $\hh^{+}$. 
    These two domains are diffeomorphic via a Cayley-type transformation.
    We first
    study the geometrical structure of the groups of M\"obius
    transformations of  $\Delta_{\hh}$ and  $\hh^{+}$ 
    and identify original ways of 
    representing them in terms of two (isomorphic) groups of matrices with quaternionic entries. We then define the
    cross-ratio of four quaternions, prove that, when real, it is 
   invariant under the action of 
   the M\"obius transformations, and use it to define
    the analogous of the Poincar\'e distances on $\Delta_{\hh}$ and  $\hh^{+}$. 
   We easily deduce that there exists no isometry between the quaternionic 
    Poincar\'e distance of $\Delta_{\hh}$
and the Kobayashi distance inherited by $\Delta_{\hh}$
  as a domain of $\mathbb{C}^{2}$, in accordance with a direct consequence 
  of the classification of the non compact, rank 1,  symmetric spaces.
\end{abstract}

\maketitle

\section{Introduction}

The study of the intrinsic geometry of the open unit disc $\mathbb{D}=\{z 
\in \mathbb{C} : |z|<1 \}$ of the 
complex plane, bi-holomorphic via the Cayley
transformation to the upper half-plane $\Pi^{+}$ of $\mathbb{C}$, is very 
rich and of great, 
classical interest. The main tool for the study of this geometry is the 
Poincar\'e distance, which turns out to be the integrated distance of 
the Poincar\'e differential metric. In fact the holomorphic self-maps 
of $\mathbb{D}$ reveal to be contractions, and
hence the group of all holomorphic automorphisms of $\mathbb{D}$ are
isometries, for the Poincar\'e distance (and differential metric). 
As a consequence, an approach typical
of differential geometry can be adopted to study the geometric theory 
of holomorphic self-maps of any simply connected domain strictly contained in $\mathbb{C}$. In fact, 
by the Riemann representation theorem, any such domain is bi-holomorphic to $\mathbb{D}$
(and to the upper half plane $\Pi^{+}$). In this setting the
classical groups $SU(1,1)$ and $SL(2, \mathbb{R})$ come into the
scenary: when quotiented by their centers, they represent the group of 
all holomorphic automorphisms (the so called
M\"obius transformations) of $\mathbb{D}$ and $\Pi^{+}$, respectively.

When endowed with the Poincar\'e differential metric, the open, unit disc $\mathbb{D}$ acquires a structure of Riemannian
surface of constant negative curvature, whose geodesics are the arcs of 
circles or straight lines which intersect the
boundary $\partial\mathbb{D}$ orthogonally.

It is interesting to notice that the Poincar\'e distance can be
defined on $\mathbb{D}$ by means of the family of all geodesics mentioned 
above. Following the approach of Siegel, \cite{Siegel}, given two
points $z_{1}, z_{2} \in \mathbb{D}$ one can define the two ends $z_{3},
z_{4}$ of the
(unique) geodesic passing through $z_{1}$ and $z_{2}$ as the
intersections of this geodesic with $\partial\mathbb{D}$. One then orders the
four points
``cyclically'' and defines the Poincar\'e distance $\delta_{\mathbb{D}}(z_{1},z_{2})$ as half the 
logarithm of the cross-ratio of the four points
$z_{1},z_{2},z_{3},z_{4}$.

With this in mind, in the present paper we consider the space $\hh$ of quaternions  and study the geometry of the open, unit disc  $\Delta_{\hh}=\{ q\in
\hh : |q|<1\}$ and of
the half-space $\hh^{+}=\{ q\in \hh : \Re e(q)>0 \}$, which turn out to be diffeomorphic 
via a Cayley-type transformation.
More precisely, we give this paper a double aim. The first one is to study the groups of 
{\em M\"obius transformations} of  $\Delta_{\hh}$ and
of $\hh^{+}$ (i.e. the groups of all quaternionic,
fractional, 
linear transformations which leave $\Delta_{\hh}$ and $\hh^{+}$
invariant, respectively). The second  aim is to give a direct, geometric definition of the
analogue  of the Poincar\'e distance (i.e. the real, hyperbolic distance) and differential metric in the
quaternionic setting, to investigate their most interesting
properties, and to explicitly describe the invariant geometry of the
classical hyperbolic domains $\Delta_{\hh}$ and $\hh^{+}$ of $\hh$. 

In section 2, in order to identify the group of all quaternionic, fractional, linear
transformations, we start by studying the problem of finding the inverse 
of a quaternionic 
$2 \times 2$ matrix. This problem 
corresponds to solve, when possible, a linear system of four quaternionic
equations, and leads us to define, in a direct and very natural way, the
so called Dieudonn\'e determinant of a quaternionic $2 \times 2$
matrix:

\begin{definiz} If $A= \left[ \begin{array}{ll}
             a & b \\
             c & d \\
             \end{array} \right]$ is a $2 \times 2$ matrix with quaternionic entries, then the {\it
(Dieudonn\'e) determinant} of $A$ is defined to be the non negative
real number 
 \begin{eqnarray}
 det_{\mathbb{H}}(A)= \sqrt{|a|^2|d|^2 + |c|^2|b|^2
-2\Re e(c\overline{a}b\overline{d})}.
 \end{eqnarray}
\end{definiz}

\noindent The notion of quaternionic determinant appears in the literature in a
much more general
setting and uses at that level the tool of quasideterminants, 
\cite{GRW, WR}.
Here we study the principal properties of this determinant - giving simple, direct proofs  
of our
assertions - also for the sake of
completeness (see also \cite{C-D, As}). We then exploit these
properties
to investigate the structure of the group of all quaternionic,
fractional, linear
transformations of $\hh$. In fact in section 3 we set $\mathbb{G}=\{
g(q)=(aq+b)(cq+d)^{-1} : a,b,c,d \in \hh, \,\, g \ 
\textnormal{invertible}\  \}$, $GL(2, \hh)=\{A 
\ \ \ 
2\times 2 \ \ \ \textnormal{matrix with quaternionic entries} : det_{\hh}(A)\neq 0\}$ and we prove

\begin{teo}
The set $\mathbb{G}$ of all quaternionic, fractional, linear
transformations 
is a group with respect to composition. The map 
\begin{eqnarray}
   \Phi: A=\left[ \begin{array}{ll}

             a & b \\
             c & d \\
             \end{array} \right]\mapsto
L_{A}(q)=(aq+b)\cdot(cq+d)^{-1}
    \end{eqnarray}
    is a group homomorphism of $GL(2, \mathbb{H})$ onto $\mathbb{G}$
    whose kernel is the center of $GL(2, \mathbb{H})$, that is the
    subgroup 
    $$\left\{ \left[ \begin{array}{ll}

             t & 0 \\
             0 & t \\
             \end{array} \right] : t \in \mathbb{R}\backslash\{0\}
\right\}.
	     $$
\end{teo}

In section 4 we extend the structure-theorem of the complex,
fractional, linear transformations to the quaternionic environment and
prove that the group $\mathbb{G}$ is generated by all the
similarities, 
$L(q)=aq+b$ $(a, b \in \mathbb{H}, a\neq 0)$  
and the inversion $R(q)=q^{-1}.$ Moreover, all the elements of 
$\mathbb{G}$ turn out to be conformal.

If the role of the complex cross-ratio is crucial in 
complex, projective geometry, its (real) generalizations to higher dimensions
in $\rr^{n}$ seem not to have a minor role in 
conformal geometry. In fact L. Ahlfors, while studying the conformal structure 
of $\rr^{n}$, has given in \cite{Ah3} 
three different definitions of the cross-ratio of $4$ points of $\mathbb{R}^n$.  The one
that we give here, specialized to the quaternionic case, is new also
with respect to the ones given by Ahlfors. In fact our definition of
cross-ratio has the peculiar
feature that the quaternionic, fractional, linear transformations act on
it transforming its value by (quaternionic) conjugation (see corollary 
\ref{crreale}). We prove, in particular, that

\begin{prop}\label{1.4}
Let $ 
\ratio(q_1,q_2,q_3,q_4):=(q_1-q_3)(q_1-q_4)^{-1}(q_2-q_4)(q_2-q_3)^{-1}$
be the
cross-ratio of the four quaternions $q_1, q_2, q_3, q_4$.  
When the cross ratio of four
quaternions is real, then it is invariant under the action of all
quaternionic, 
fractional, linear transformations.
\end{prop}

The above result has a great deal of interest in view of the fact that

\begin{prop} \label{1.5}
Four pairwise distinct points $q_1,$$q_2,$$q_3,$$q_4 \in \mathbb{H}$
lie on a same
(one-dimensional) circle or straight line if, and only if, their cross-ratio
is real. 
\end{prop}

When A.F. M\"{o}bius introduced the notion of what we call nowadays a 
fractional, linear
transformation, 
what he had in mind was only a 
homemorphism of the extended, complex plane $\mathbb{C} \cup \{ \infty \}$ 
onto itself which maps 
circles onto circles. Adopting this point of view, still in section
\ref{quattro} we define the families $\mathcal{F}_i,$ for $i=3,2,1,$ respectively 
as $\mathcal{F}_i=\mathcal{S}_i \cup \mathcal{P}_i$ where $\mathcal{S}_i$ 
is the family of all $i-$(real) dimensional spheres and $\mathcal{P}_i$ is the family 
of all $i-$(real) dimensional affine subspaces of $\mathbb{H}.$ Then
we give an original proof of the fact that

\begin{teo}
The  group  $\mathbb{G}$ of all quaternionic, fractional, linear transformations maps
elements of $\mathcal{F}_i$ onto elements of  $\mathcal{F}_i$, for
$i=3,2,1$.
\end{teo}

The aim of section 5 is to find a geometric approach to the definition of the quaternionic Poincar\'e distance on  
$\Delta_{\hh}$ (often simply called Poincar\'e distance when no confusion can arise). To this aim we adopt the point of view used by C. L. Siegel,
\cite{Siegel}, for the homologous problem in the complex case and use the terminology introduced by Ahlfors in \cite{Ah3} . In fact,
 to start with,  we define the non-Euclidean line through two points
$q_{1}, q_{2}$ as
the unique circle, or diameter, containing the two
points and intersecting $\partial
\Delta_{\hh}$ orthogonally in the
two ends  $q_{3}, q_{4}$. The Poincar\'e distance of  $\Delta=\Delta_{\hh}$ is then defined by 
\begin{equation}\label{poincintr}
\delta_{\Delta}(q_{1},
q_{2})=\frac{1}{2}\log (\ratio(q_{1}, q_{2}, q_{3}, q_{4}))
\end{equation}
where 
the four points are arranged cyclically on the non-Euclidean line
through $q_{1}$ and $q_{2}$. Notice that on each complex plane
$L_{I}=\rr+I\rr$ (for any imaginary unit $I$) the quaternionic Poincar\'e
distance coincides with the classical Poincar\'e distance of $\Delta_{I}=\Delta_{\hh}\cap L_{I}$. 

The structure of the group
$\mathbb{M}$ of  M\"{o}bius transformations
 of  $\Delta_\mathbb{H}$
 is studied, for example,  in  \cite{CPW}, in terms of the (classical) group $Sp(1, 1)$.
If $H=\left[ \begin{array}{rr}

             1 & 0 \\
             0 & -1 \\
             \end{array} \right]$, the group $Sp(1,1)$ is defined (see, e.g., \cite{GOV}) as 
\begin{equation}\label{sp}
    Sp(1,1)=\left\{  A\in GL(2, \mathbb{H})\  :\ \  ^{t}\overline{A}HA=H 
    \right\}
\end{equation}
and it can be written equivalently as (see, e.g., \cite{CPW})
\begin{equation*}
    Sp(1,1)=\left\{  \left[ \begin{array}{ll}

             a & b \\
             c & d \\
             \end{array} \right] : |a|=|d|, \ \ |b|=|c|, \ \  |a|^2 -|c|^2 =1, 
	     \ \ 
\overline{a}b=\overline{c}d, \ \ a\overline{c}=b\overline{d}
    \right\}.
\end{equation*}
It allows to rephrase and complete a result of \cite{CPW} as
follows:
\begin{teo} \label{nice}
The quaternionic, fractional, linear transformation defined by $g(q)=(aq+b)(cq+d)^{-1}$  is a M\"obius 
transformation of $ \Delta_{\mathbb{H}}$ if and only if $\left[ \begin{array}{ll}

             a & b \\
             c & d \\
             \end{array} \right] \in Sp(1,1)$.
  Moreover the map 
\begin{eqnarray}
   &\phi: Sp(1,1) \to \mathbb{M} \nonumber \\
   &A=\left[ \begin{array}{ll}
             a & b \\
             c & d \\
             \end{array} \right]\mapsto
L_{A}(q)=(aq+b)\cdot(cq+d)^{-1}
    \end{eqnarray}
    is a group homomorphism 
    whose kernel is the center of $Sp(1,1)$, that is the
    subgroup 
    $$\left\{\pm \left[ \begin{array}{ll}

             1& 0 \\
             0 & 1\\
             \end{array} \right] 
\right\}.
	     $$
\end{teo}

\noindent By means of the statement of theorem \ref{nice} we are
able to obtain, for the
quaternionic 
M\"obius transformations, a characterization  which closely resembles
the classical
representation of the complex M\"obius transformations. A similar
result is stated without proof in \cite{HJ}.
\begin{teo}\label{ultimo}
Each quaternionic M\"{o}bius transformation
$g(q)=(aq+b)\cdot(cq+d)^{-1}\in \mathbb{M}$
can be
written uniquely as:
\begin{equation} 
g(q)=\alpha (q - q_0)(1-\overline{q_0}q)^{-1}  \beta^{-1}
\end{equation}
where $q_0 = -a^{-1} b \in \Delta_{\mathbb{H}}$ and
where
$\alpha =\dfrac{a}{|a|}\in \partial \Delta_{\mathbb{H}}$, 
$\beta = \dfrac{d}{|d|} \in \partial \Delta_{\mathbb{H}}.$ 
\end{teo}
\noindent The description of the group of all M\"obius 
transformations of $\Delta_{\hh}$ given in theorem \ref{ultimo} is different from the one
given in a more general setting in \cite{Ah3}.
\noindent Using propositions \ref{1.4} and \ref{1.5} we sum up by proving the following

\begin{prop} The Poincar\'e distance of
    $\Delta_{\mathbb{H}}$ is invariant under the action of the group
    of all  M\"obius transformations $\mathbb{M}$ and of the map $q
    \mapsto \overline{q}$. 
\end{prop}

It is now possible to mimic 
the definition of the classical, complex
Poincar\'e differential metric of $\mathbb{D} \subset \mathbb{C}$ 
to set the length of the vector $\tau
\in \hh$  for the Poincar\'e metric at $q\in \Delta_{\hh}$ to be the
number:
\begin{equation}\label{poinmetricintr}
\langle \tau \rangle_{q} = \frac{|\tau|}{1-|q|^{2}}.
\end{equation}
Formula (\ref{poinmetricintr}) leads now to 
the definition of the (square of the) Poincar\'e length element at $q\in
\Delta_{\hh}$:
\begin{equation*}
ds^2= \frac{|d_I q|^2}{(1-|q|^2)^2}
\end{equation*}
where $q=x+yI$ and $d_I q= dx + I dy$ (for $I\in \mathbb{S}$).
The quaternionic Poincar\'e differential metric given above can also
be obtained by specializing to the case of quaternions the definition 
given in the more general setting of the study of conformal geometry
of $\rr^{n}$ by Ahlfors, \cite{Ah3}.
At the end of section 5, the following results are proved:

\begin{teo} All the elements of
    the group $\mathbb{M}$ of  M\"obius transformations  of
    $\Delta_{\hh}$, as well as the map $q \mapsto \overline{q}$, 
    leave the Poincar\'e differential metric
    invariant.
\end{teo}

\begin{prop}
    The Poincar\'e distance $\delta_{\Delta}$ of the unit disc $\Delta_{\hh}$
    is the integrated distance of the Poincar\'e differential metric
    of $\Delta_{\hh}$.
\end{prop}

It is easy, at this point, to deduce directly that the
invariant, metric structure defined by the quaternionic Poincar\'e distance 
(and metric) and
the one induced  by the Kobayashi distance (and
metric) on $\Delta_{\hh}\cong \Delta_{\mathbb{C}^{2}}$ (see \cite{FV}, 
\cite{Rud}) are not isometric. 
We do this in section 6, where, in accordance with a consequence of 
the classification of non compact, rank 1, symmetric spaces (see, e.g., \cite{DNF}, \cite{H}), 
we state and prove that:

\begin{teo}
There exists no isometry between the quaternionic Poincar\'e distance
and the Kobayashi distance of $\Delta_{\mathbb{H}}\cong
\Delta_{\mathbb{C}^{2}}$.
\end{teo}

Section 7 is dedicated to transfer the Poincar\'e distance and
differential metric of $\Delta_{\mathbb{H}}$ to 
$\mathbb{H}^+$ via a
Cayley-type transformation. The results obtained in $\hh^{+}$ are
 homologous to those which hold in $\Delta_{\hh}$.
Nevertheless, in this setting,  we are able to give an original, nice description of 
the group of all M\"obius transformations $\mathbb{M}(\hh^{+})$ of
$\hh^{+}$, in terms of a group  of matrices $\GLH$ which plays the role
played by the group $SL(2, \rr)$  in the complex case.

\begin{teo} If $K=\left[ \begin{array}{ll}

              0 & 1 \\
              1 & 0 \\
             \end{array} \right] $, 
    then the set of matrices defined by 
     \begin{equation*} 
	     \GLH = \left\{   A\in GL(2,\mathbb{H}) :\ \ 
	     ^{t}\overline{A}KA=K \right\}
\end{equation*}
is a subgroup of $SL(2, \hh)$ of real dimension 10.  Moreover, 
\begin{equation*}
\GLH = \left\{   \left[ \begin{array}{ll}

              a & b \\
              c & d \\
             \end{array} \right] : a,b,c,d \in \mathbb{H},  
	    \,\, \Re e (a\overline{c})=0,\,\,\, \Re e
(b\overline{d})=0, \,\,\, \overline{b}c+\overline{d}a=1  \right\}.
\end{equation*}

\noindent The map 
\begin{eqnarray}
   &\Psi: \GLH \to \GL\nonumber \\
   &A=\left[ \begin{array}{ll}
             a & b \\
             c & d \\
             \end{array} \right]\mapsto
L_{A}(q)=(aq+b)\cdot(cq+d)^{-1}
    \end{eqnarray}
    is a group homomorphism 
    whose kernel is the center of $\GLH$, that is the
    subgroup 
    $$\left\{\pm \left[ \begin{array}{ll}

             1& 0 \\
             0 & 1\\
             \end{array} \right] 
\right\}.
	     $$
\end{teo}
\noindent The last result of this paper states that 
\begin{teo} The two subgroups $\GLH$ and $Sp(1,1)$ of $SL(2, \mathbb{H})$ are isomorphic.
   \end{teo}

In what follows, the
elements  of the skew field $\mathbb{H}$ of real quaternions will be denoted by $q=x_0+ix_1+jx_2+kx_3$ where the $x_l$
are real, and $i$, $j$, $k$, are imaginary units (i.e. their
square equals $-1$) such that $ij=-ji=k$, $jk=-kj=i$, and
$ki=-ik=j.$ We will denote by $\mathbb{S}_{\mathbb{H}}^{3}$ the 
sphere of quaternions of unitary modulus $ \{q \in \mathbb{H} : |q|=1\}$ and by $\mathbb{S}$ the unit sphere 
of purely imaginary
quaternions, i.e. $\mathbb{S}=\{q=ix_1+jx_2+kx_3 : x_1^2+x_2^2+x_3^2=1\}.$ Notice that if $I\in\mathbb{S}$, then
$I^2=-1$; for this reason the elements of $\mathbb{S}$ are called
imaginary units.  We will also often use the fact that for any non-real 
quaternion $q\in \mathbb{H} \backslash
\mathbb{R}$, there exist, and are unique, $x, y \in \mathbb{R}$ with $y>0$, 
and $I\in \mathbb{S}$ such that $q=x+yI$.

\section{The determinant of $2\times 2$ matrices with quaternionic
entries}

As it is well known, the determinant of a matrix with quaternionic
entries cannot be defined as in the case of matrices with real or
complex entries.  Nevertheless, the study of the quaternionic
analogue of the fractional, linear and M\"{o}bius, complex
transformations leads us to an interesting generalization of the
notion of determinant, in the case of  $2\times 2$ quaternionic
matrices.  
The notion of quaternionic determinant appears in the literature in a
much more general
setting and exploits at that level the tool of quasideterminants, 
\cite{GRW, WR}.
Here we will present the main features of the determinant of $2\times
2$ 
quaternionic
matrices - giving simple, direct proofs  of our
assertions - for the sake of
completeness (see also \cite{C-D, As}).

We will denote by $M(2,\mathbb{H})$ the $\mathbb{H}$-vector space
(right or left, depending on
the setting) of
$2\times 2$ matrices with quaternionic entries and by $GL(2,
\mathbb{H})$ the group of invertible elements of $M(2,\mathbb{H})$. A
matrix
$ \left[ \begin{array}{ll}

             a & b \\
             c & d \\
             \end{array} \right] \in M(2, \mathbb{H})
$ 
is invertible if and only if there exists 
$ \left[ \begin{array}{ll}

             x & y \\
             t & z \\
             \end{array} \right] \in M(2, \mathbb{H})
$ 
such that     
\begin{equation}
 \left[ \begin{array}{ll}

             a & b \\
             c & d \\
             \end{array} \right] \cdot 
              \left[ \begin{array}{ll}

             x & y \\
             t & z \\
             \end{array} \right] =
              \left[\begin{array}{ll}

             1 & 0 \\
             0 & 1 \\
             \end{array} \right]
\end{equation} 
i.e., if and only if the following system of linear equations 
\begin{equation}\label{sistema}
 \left\{ \begin{array}{lll}

             ax+bt & = & 1 \\
            cx+dt & = & 0\\
             ay+bz & = & 0 \\
            cy+dz & = & 1\\
             \end{array} \right.
 \end{equation} 
has a (unique) solution $(x,y,t,z)\in \mathbb{H}^4$. We can now prove

\begin {prop}\label{alternativa} The following three statements are
equivalent:
\begin{enumerate}
\item the matrix $A= \left[ \begin{array}{ll}

             a & b \\
             c & d \\
             \end{array} \right] \in M(2, \mathbb{H})
$ 
is invertible; 
\item $b(c-db^{-1}a)\ne 0$ \textnormal{or} $a(d-ca^{-1}b)\ne 0;$
\item $c(b-ac^{-1}d)\ne 0$ \textnormal{or }$d(a-bd^{-1}c)\ne 0.$ 
\end{enumerate}
\end{prop}
\begin{proof} We will begin by proving that (1) implies (2). The
first equation of (\ref{sistema}) implies that $a\ne 0$ or $b\ne 0$.
If $a\ne 0$ then, using the third equation in (\ref{sistema}), we
obtain  $y=-a^{-1}bz$ and substituting in the fourth equation of the
same system we get  $(d-ca^{-1}b)z=1$. Therefore we obtain
$(d-ca^{-1}b)\ne 0$ and $a(d-ca^{-1}b)\ne 0$. At this point an easy
computation shows that in this case
\begin{equation}\label{inversa a}
  \left[ \begin{array}{ll}
             x & y \\
             t & z \\
             \end{array} \right] =
               \left[ \begin{array}{ll}
            a^{-1} + a^{-1}b(d-ca^{-1}b)^{-1}ca^{-1} &
-a^{-1}b(d-ca^{-1}b)^{-1} \\
             -(d-ca^{-1}b)^{-1}ca^{-1} & (d-ca^{-1}b)^{-1} \\
             \end{array} \right].
\end{equation}
If we are in the case $b\ne 0$ then, using as above system
(\ref{sistema}), we obtain $z=-b^{-1}ay$ and $(c-db^{-1}a)y=1$,
yielding $(c-db^{-1}a)\ne 0$ and hence $b(c-db^{-1}a)\ne 0$. As
before, an easy computation shows now that
\begin{equation}\label{inversa b}
  \left[ \begin{array}{ll}
             x & y \\
             t & z \\
             \end{array} \right] =
               \left[ \begin{array}{ll}
            -(c-db^{-1}a)^{-1}db^{-1} & (c-db^{-1}a)^{-1}\\
            b^{-1}+b^{-1}a(c-db^{-1}a)^{-1}db^{-1} &
-b^{-1}a(c-db^{-1}a)^{-1} \\
             \end{array} \right].
\end{equation}
To  prove that (2) implies (1), it is enough to notice that when (2)
is assumed true, matrix (\ref{inversa a}) or (\ref{inversa b}) is
well defined and that it is (by construction) the inverse of $A$.
The proof of the equivalence of (1) and (3) is completely analogous
to the one given above.
\end{proof}

\begin{nota} As one may expect, when $ab\ne 0$ then the two forms
    (\ref{inversa a}) and (\ref{inversa b}) of the inverse of $A$ do
coincide. If $abcd\neq 0$ then the inverse matrix of $A$ assumes an
even nicer
form,
\begin{equation*}
  \left[ \begin{array}{ll}
             x & y \\
             t & z \\
             \end{array} \right] =
               \left[ \begin{array}{ll}
            (a-bd^{-1}c)^{-1} & (c-db^{-1}a)^{-1} \\
            (b-ac^{-1}d)^{-1} &  (d-ca^{-1}b)^{-1} \\
             \end{array} \right]
\end{equation*}
which allows a Cramer-type rule to solve $2 \times 2$ linear systems
with
quaternionic coefficients (see also \cite{WR}).
\end{nota}

Let us now compute
\begin{eqnarray}
&
&|a(d-ca^{-1}b)|^2=a(d-ca^{-1}b)(\overline{d}-\overline{b}\overline{a}^{-1}\overline{c})
\overline{a}
\nonumber \\
&= &
a(|d|^2-d\overline{b}\overline{a}^{-1}\overline{c}-ca^{-1}b\overline{d}+|c|^2|a|^{-2}|b|^2)
\overline{a}\nonumber
\\
&=
&|a|^2|d|^2-a(2\Re
e(d\overline{b}\overline{a}^{-1}\overline{c}))\overline{a}
+ |c|^2|b|^2\nonumber \\
&=
&|a|^2|d|^2-|a|^2(2\Re e(d\overline{b}\overline{a}^{-1}\overline{c}))
+
|c|^2|b|^2\nonumber \\
&=  & |a|^2|d|^2 + |c|^2|b|^2
-2\Re e(d\overline{b}a\overline{c}).\nonumber
\end{eqnarray}
Similarly we obtain
\begin{eqnarray}
|b(c-db^{-1}a)|^2=|a|^2|d|^2 + |c|^2|b|^2
-2\Re e(c\overline{a}b\overline{d}).\nonumber
\end{eqnarray}
We analogously get
\begin{eqnarray}\label{secondo}
|c(b-ac^{-1}d)|^2=|d(a-bd^{-1}c)|^2=|a|^2|d|^2 + |c|^2|b|^2
-2\Re e(a\overline{c}d\overline{b}).
\end{eqnarray}
Since $\Re e(uv)=\Re e(\overline{u}\ \overline{v})$ for any $u,v\in
\mathbb{H}$, we also have
$\Re e(c\overline{a}b\overline{d})=\Re
e(a\overline{c}d\overline{b})$, and
therefore

\begin{lemma}\label{uguaglianza} The following equalities hold
    \begin{eqnarray}
& &|a(d-ca^{-1}b)|^2=|b(c-db^{-1}a)|^2=
|c(b-ac^{-1}d)|^2=|d(a-bd^{-1}c)|^2\nonumber\\
&=&|a|^2|d|^2 + |c|^2|b|^2
-2\Re e(c\overline{a}b\overline{d}).\nonumber
\end{eqnarray}
    \end{lemma}

\begin{nota}\label{positivo} For all $A= \left[ \begin{array}{ll}
             a & b \\
             c & d \\
             \end{array} \right]\in M(2, \mathbb{H})$, 
    it turns out that
 \begin{eqnarray}\label{determinante positivo}
 |a|^2|d|^2 + |c|^2|b|^2
-2\Re e(c\overline{a}b\overline{d})\\ \nonumber
 \ge |a|^2|d|^2 + |c|^2|b|^2
-2|a||d||b||c|\\ \nonumber
=(|a||d| - |b||c|)^{2}\ge 0.
 \end{eqnarray}
\end{nota}

Proposition \ref{alternativa}, lemma
\ref{uguaglianza} and remark \ref{positivo} naturally lead to the following 
definition, which can also be found in \cite{GRW, WR}.

\begin{definiz}\label{determinante} If $A= \left[ \begin{array}{ll}
             a & b \\
             c & d \\
             \end{array} \right] \in M(2, \mathbb{H})$, then the {\it
(Dieudonn\'e) determinant} of $A$ is defined to be the non negative
real number 
 \begin{eqnarray}\label{determinante 1}
 det_{\mathbb{H}}(A)= \sqrt{|a|^2|d|^2 + |c|^2|b|^2
-2\Re e(c\overline{a}b\overline{d})}.
 \end{eqnarray}
\end{definiz}

\begin{nota} It is worthwhile noticing that when $A= \left[
\begin{array}{ll}

             a & b \\
             c & d \\
             \end{array} \right]$
has complex (or real) entries, then $det_{\mathbb{H}}(A)=
|ad-bc|=|det(A)|$, i.e, the
new notion of determinant coincides with the modulus of 
the classical determinant.
\end{nota}

The interest of the preceeding definition is made clear by the
following

\begin{prop} A matrix $A\in M(2,\mathbb{H})$ is invertible if, and
only
if, $det_{\mathbb{H}}(A) \ne 0$.
\end{prop}
\begin{proof} The proof is a direct consequence of proposition
\ref{alternativa} 
    and lemma \ref{uguaglianza}.
\end{proof}

We end this section by proving that the analogue of the Binet-Cauchy
 formula holds for $det_{\mathbb{H}}$. This fact is established in a
more
 general setting in \cite{GRW, WR}, where the proof is based on the
 properties of quasideterminants and does not contain all the details.
 In any case we give here a simple proof.
 
 \begin{lemma}\label{proprieta} For any $\lambda, \mu 
    \in \mathbb{H}$ and any matrix $X=\left[ \begin{array}{ll}

             x & y \\
             z & t \\
             \end{array} \right]\in M(2, \mathbb{H})$ we have:
\begin{itemize}	\vskip .2cm    
	     \item[i)] $\label{lambda}
	     det_{\mathbb{H}}\left[ \begin{array}{ll}
	     
             x & y\lambda \\
             z & t\lambda \\
             \end{array} \right] = 
	     det_{\mathbb{H}}\left[ \begin{array}{ll}
	     
             x\lambda  & y \\
             z\lambda  & t \\
             \end{array} \right]=|\lambda|det_{\mathbb{H}}
	     \left[ \begin{array}{ll}

             x & y \\
             z & t\\
             \end{array} \right]
	     $ \vskip .4cm
\item[ii)] $
det_{\mathbb{H}}\left[ \begin{array}{ll}
	     
             \mu x & \mu y \\
             z & t \\
             \end{array} \right] = 
	     det_{\mathbb{H}}\left[ \begin{array}{ll}
	     
             x  & y \\
             \mu z  & \mu t \\
             \end{array} \right]=|\mu|det_{\mathbb{H}}
	     \left[ \begin{array}{ll}

             x & y \\
             z & t\\
             \end{array} \right]
	    $ \vskip .3cm
 \item[iii)] If the matrix $Y$ is obtained from the matrix $X$ by: (a)
substituting to a row the sum of the two rows, or (b) substituting to
a column the sum of the two columns, then
$det_{\mathbb{H}}(X)=det_{\mathbb{H}}(Y).$
\end{itemize}
 \end{lemma}
 \begin{proof} A direct substitution and computation show the
     assertions.
     \end{proof}

\begin{prop}[Binet property]\label{Binet} For all $A, B\in M(2,
\mathbb{H})$ we have
    that
$det_{\mathbb{H}}(AB)=det_{\mathbb{H}}(A)det_{\mathbb{H}}(B)$.
\end{prop}
\begin{proof} We can suppose $A, B$ invertible (otherwise
    the proof is immediate).
    If $A=\left[ \begin{array}{ll}
             a & b \\
             c & d \\
             \end{array} \right]$ and
	     $B=\left[ \begin{array}{ll}
             e & f \\
             g & h \\
             \end{array} \right]$,
then $AB=\left[ \begin{array}{ll}

             ae+bg & bh+af \\
             ce+dg & dh+cf \\
             \end{array} \right]$. 
	     
We will operate now on the matrix $AB$ step by step, and use lemma
\ref{proprieta} at each step, to compute its determinant. If $h\neq 0$
then we have:
$$det_{\mathbb{H}}(AB)=det_{\mathbb{H}}\left[ \begin{array}{ll}

             ae+bg & bh+af \\
             ce+dg & dh+cf \\
             \end{array} \right]
             $$
             $$
=
det_{\mathbb{H}}\left[ \begin{array}{ll}

             ae+bg -  (bh+af)h^{-1}g & bh+af \\
             ce+dg - (dh+cf)h^{-1}g & dh+cf \\
             \end{array} \right]
$$
by lemma \ref{proprieta} i), iii). Now
$$
\left[ \begin{array}{ll}

             ae+bg -  (bh+af)h^{-1}g & bh+af \\
             ce+dg - (dh+cf)h^{-1}g & dh+cf \\
             \end{array} \right]
	     =
	     \left[ \begin{array}{ll}

             a(e-fh^{-1}g) & bh+af \\
             c(e-fh^{-1}g) & dh+cf \\
             \end{array} \right]
$$
 and
$$det_{\mathbb{H}}\left[ \begin{array}{ll}

             a(e-fh^{-1}g) & bh+af \\
             c(e-fh^{-1}g) & dh+cf \\
             \end{array} \right]=
             $$
             $$
	     det_{\mathbb{H}}\left[ \begin{array}{ll}

             a(e-fh^{-1}g) & bh+af-a(e-fh^{-1}g)(e-fh^{-1}g)^{-1}f  \\
             c(e-fh^{-1}g) & dh+cf-c(e-fh^{-1}g)(e-fh^{-1}g)^{-1}f \\
             \end{array} \right]$$
again by lemma \ref{proprieta} i), iii) and since $B$ is invertible
(see proposition \ref{alternativa}) . We have 
$$
\left[ \begin{array}{ll}

             a(e-fh^{-1}g) & bh+af-a(e-fh^{-1}g)(e-fh^{-1}g)^{-1}f  \\
             c(e-fh^{-1}g) & dh+cf-c(e-fh^{-1}g)(e-fh^{-1}g)^{-1}f \\
             \end{array} \right]
             $$
             $$
	     =
	     \left[ \begin{array}{ll}

             a(e-fh^{-1}g) & bh  \\
             c(e-fh^{-1}g) & dh \\
             \end{array} \right]
$$
and, by lemma \ref{proprieta} i), iii),
$$
det_{\mathbb{H}}\left[ \begin{array}{ll}

             a(e-fh^{-1}g) & bh  \\
             c(e-fh^{-1}g) & dh \\
             \end{array} \right]
             $$
             $$
	     =
	     det_{\mathbb{H}}\left[ \begin{array}{ll}

             a & b  \\
             c & d\\
             \end{array}
	     \right]|e-fh^{-1}g)h|=
	     det_{\mathbb{H}}(A)det_{\mathbb{H}}(B).
	     $$
In the remaining case in which $h=0$, the coefficient $f$ does not
vanish and the matrix $AB$ becomes
$$
\left[ \begin{array}{ll}

             ae+bg & af  \\
             ce+dg & cf \\
             \end{array} \right].
$$
Then,  by lemma \ref{proprieta} i), iii), 
$$
det_{\mathbb{H}}(AB)=det_{\mathbb{H}}\left[ \begin{array}{ll}

             ae+bg & af  \\
             ce+dg & cf \\
             \end{array} \right]
	     =
	     det_{\mathbb{H}}\left[ \begin{array}{ll}

             af(f^{-1}e)+bg & af  \\
             cf(f^{-1}e)+dg & cf \\
             \end{array} \right]
	     =$$
	     $$
	     det_{\mathbb{H}}\left[ \begin{array}{ll}

             bg & af  \\
             dg & cf \\
             \end{array}\right]
	     =|g||f|det_{\mathbb{H}}(A)=
	     det_{\mathbb{H}}(A)det_{\mathbb{H}}(B).
$$
\end{proof}

\section{Fractional linear transformations and their properties}

For any matrix $A=\left[ \begin{array}{ll}

             a & b \\
             c & d \\
             \end{array} \right] \in M(2, \mathbb{H})$, with $c\ne 0$ 
	     or $d\ne 0$, the map 
$$L_{A}(q)=(aq+b)\cdot(cq+d)^{-1}$$
is called a {\it (quaternionic) fractional linear map}. To identify
constant 
maps, we
will give the following characterization:

\begin{prop}\label{non costante}
The fractional linear map  $L_{A}(q)=(aq+b)\cdot(cq+d)^{-1}$ is
constant if, and only if, $det_{\mathbb{H}}(A)=0.$
\end{prop}
\begin{proof1} If the fractional linear transformation $L_{A}$ is
constant, 
i.e. if $L_{A}(q)=k,$ for all $q\in \mathbb{H}$, then 
$$
(aq+b)\cdot(cq+d)^{-1}=k
$$
$$
aq+b= kcq+kd
$$
$$
(a-kc)q=kd - b
$$
for all $q\in \mathbb{H}$. Thus
\begin{equation}
 \left\{ \begin{array}{lll}
             a-kc & = & 0 \\
            kd-b & = & 0\\
             \end{array} \right.
 \end{equation} 
yielding $$A= \left[ \begin{array}{ll}

             kc & kd \\
             c & d \\
             \end{array} \right]
	     $$
and $det_{\mathbb{H}}(A)=0$.

\noindent Conversely, if $det^{2}_{\mathbb{H}}(A)=
|a|^2|d|^2 + |c|^2|b|^2 -2\Re e(c\overline{a}b\overline{d})=0$ and
$abcd=0$ 
then $cb=0$ or $ad=0$. Since in this case $det^{2}_{\mathbb{H}}(A)=
|a|^2|d|^2$ or $det^{2}_{\mathbb{H}}(A)=
|c|^2|b|^2$, we obtain that $cb=0$ and $ad=0$. If $c=0$ then, by
definition, $d\ne 0$ and hence $a=0$, yielding $L_{A}(q)=bd^{-1}$ for 
all $q\in \mathbb{H}$. On the other hand, if $b=0$ then, either $a=0$
and
$L_{A}\equiv 0$, or $d=0$ implying $L_{A}(q)=ac^{-1}$ for 
all $q\in \mathbb{H}$. To conclude the proof, we notice that when
$abcd\ne 0$, then by proposition \ref{uguaglianza} we obtain for
example $c=db^{-1}a,$ which leads to 
\begin{eqnarray}
L_{A}(q)= (aq+b)(db^{-1}aq
+d)^{-1}=(aq+b)[db^{-1}(aq+bd^{-1}d)]^{-1}\nonumber\\
=(aq+b)(aq+b)^{-1}bd^{-1}=bd^{-1}\nonumber
\end{eqnarray}
for all $q\in \mathbb{H}$.
\end{proof1}

In analogy with the case of the complex plane $\mathbb{C}$, we give
the following

\begin{definiz}\label{defFLT}
For any matrix $A=\left[ \begin{array}{ll}

             a & b \\
             c & d \\
             \end{array} \right] \in M(2, \mathbb{H})$, the map 
$$L_{A}(q)=(aq+b)\cdot(cq+d)^{-1}$$
is called a {\it (quaternionic) fractional linear transformation} if
$det_{\mathbb{H}}(A) \ne 0$ 
i.e., if $A\in GL(2, \mathbb{H}).$
\end{definiz}

\begin{teo}\label{homo}
The set $\mathbb{G}$ of all quaternionic fractional linear
transformations 
is a group with respect to composition. The map 
\begin{eqnarray}
   \Phi: A=\left[ \begin{array}{ll}

             a & b \\
             c & d \\
             \end{array} \right]\mapsto
L_{A}(q)=(aq+b)\cdot(cq+d)^{-1}
    \end{eqnarray}
    is a group homomorphism of $GL(2, \mathbb{H})$ onto $\mathbb{G}$
    whose kernel is the center of $GL(2, \mathbb{H})$, that is the
    subgroup 
    $$\left\{ \left[ \begin{array}{ll}

             t & 0 \\
             0 & t \\
             \end{array} \right] : t \in \mathbb{R}\backslash\{0\}
\right\}.
	     $$
\end{teo}
\begin{proof1} It is a straightforward computation to prove that, 
if $L_1, L_2 \in \mathbb{G}$ 
are such that $\Phi(A_{1})=L_{1}$ and $\Phi(A_{2})=L_{2}$ for some
$A_{1}, A_{2} \in GL(2, \mathbb{H})$, then $\Phi(A_{1} \cdot
A_{2})=L_1 \circ
L_2$.
Moreover $\Phi(I_{2})=Id$ is the identity map. As a consequence,
$\Phi$ is a surjective homomorphism, and hence
$\mathbb{G}$ is a group. 

Now $L_{A}(q)=(aq+b)\cdot(cq+d)^{-1}=q$ for all $q\in \mathbb{H}$, 
if, and only if, $qcq +qd -aq -b = 0$ for all $q\in \mathbb{H}$ and
hence $c=0=b$ and $a=d\in \mathbb{R}$. The last assertion follows
immediately.
\end{proof1}

If we set $SL(2, \mathbb{H})= \{A \in GL(2, \mathbb{H}) :
det_{\mathbb{H}}(A)=1 \}$ then, as an application of the Binet
formula 
(see proposition \ref{Binet}), 
we obtain that  
$SL(2, \mathbb{H})$ is a subgroup of $GL(2, \mathbb{H})$ and that

\begin{corollario}\label{homoSl} The map 
\begin{eqnarray}
   \Phi: A=\left[ \begin{array}{ll}

             a & b \\
             c & d \\
             \end{array} \right]\mapsto
L_{A}(q)=(aq+b)\cdot(cq+d)^{-1}
    \end{eqnarray}
    is a group homomorphism of $SL(2, \mathbb{H})$ onto $\mathbb{G}$
    whose kernel is the center of $SL(2, \mathbb{H})$, that is the
    subgroup 
    $$\left\{ \pm \left[ \begin{array}{ll}

             1 & 0 \\
             0 &  1 \\
             \end{array} \right] \right\}.
	     $$
\end{corollario}
\begin{proof} The proof relies upon the fact that, for all $t\in
    \mathbb{R}\backslash\{0\}$ and all $A\in GL(2, \mathbb{H})$, we
have 
    $det_{\mathbb{H}}(tA)=t^{2}det_{\mathbb{H}}(A)> 0$.
\end{proof}

In view of corollary \ref{homoSl}, from now on we will always suppose 
that the matrix $A=\left[ \begin{array}{ll}

             a & b \\
             c & d \\
             \end{array} \right]$ associated to the fractional linear 
	     transformation 
	     $L_{A}(q)=(aq+b)\cdot(cq+d)^{-1}$ belongs to $SL(2,
	     \mathbb{H})$, unless otherwise specified.

\section{The quaternionic cross-ratio}\label{quattro}

We will generalize the classical definition 
of complex cross-ratio to the non commutative case of the Hamilton
numbers, and study its peculiar properties. 

\begin{prop}\label{0 1 infinito}
Given three distinct $\alpha,\beta,\gamma \in \mathbb{H},$ the
fractional linear 
transformation defined by
$$
(\gamma - \beta)(\gamma - \alpha)^{-1}(q-\alpha)(q-\beta)^{-1}
$$
maps $\alpha$ to $0$, $\beta$ to $\infty$ and $\gamma$  to $1$.
Moreover
all fractional linear transformations with the same property are of
the form:
$$
k(\gamma - \beta)(\gamma - \alpha)^{-1}(q-\alpha)(q-\beta)^{-1}k^{-1}
$$
with $k$ any element of $\mathbb{H}\backslash \{0\}$. 
\end{prop}
\begin{proof1}
Let us consider a generic element of $\mathbb{G}$ defined by 
$L_{A}(q)=(aq+b)(cq+d)^{-1}$, and require that
$L_{A}(\alpha)=(a\alpha+b)(c\alpha+d)^{-1}=0$,
$L_{A}(\beta)=(a\beta+b)(c\beta+d)^{-1}=\infty$ and 
$L_{A}(\gamma)=(a\gamma+b)(c\gamma+d)^{-1}=1.$ It follows that
$$
 \left\{ \begin{array}{lll}

             a\alpha+b & = & 0\\
            c\beta+d & = & 0\\
             (a\gamma+b) & = & (c\gamma+d) \\
             \end{array} \right.
 $$
 $$
  \left\{ \begin{array}{lll}

             b & = & -a\alpha\\
            d & = & -c\beta\\
             a(\gamma- \alpha) & = & c(\gamma-\beta) \\
             \end{array} \right.
 $$
$$
\left\{ \begin{array}{lll}
             a & = & c(\gamma - \beta)(\gamma - \alpha)^{-1}\\
            b= & = & -c(\gamma - \beta)(\gamma - \alpha)^{-1}\alpha\\
             d & = & -c\beta\\
             \end{array} \right.
$$
and therefore
$$
L_{A}(q)=[c(\gamma - \beta)(\gamma - \alpha)^{-1}q 
-c(\gamma - \beta)(\gamma - \alpha)^{-1}\alpha](cq-c\beta)^{-1}
$$
$$
=c(\gamma - \beta)(\gamma - \alpha)^{-1}(q-\alpha)(q-\beta)^{-1}c^{-1}
.$$

\end{proof1}

Inspired by an approach due to Ahlfors, \cite{Ah3}, we will now
define the cross-ratio of $4$-tuples of quaternions.
\begin{definiz}
The quaternionic {\it cross-ratio} of four points
$q_1,q_2,q_3,q_4$ in
$\mathbb{H}\cup \{\infty\}$ is defined as: 
\begin{equation*} 
\ratio(q_1,q_2,q_3,q_4):=(q_1-q_3)(q_1-q_4)^{-1}(q_2-q_4)(q_2-q_3)^{-1}.
\end{equation*}
\end{definiz}
\noindent To investigate
the behaviour of the quaternionic cross-ratio under the action of the
group of fractional, linear
transformations, we will make use of the following decomposition
lemma:

\begin{lemma}\label{struttura di G}
The group $\mathbb{G}$ is generated by the following four types of
fractional linear transformations:
\begin{itemize}
\item[i)] $L_1(q)=q+b,$ $\,\,\,$ $ b \in \mathbb{H};$ 
\item[ii)] $L_2(q)=aq,$ $ \,\,\,$ $a \in \mathbb{H},$ $\,\,\,$
$|a|=1$; 
\item[iii)] $L_3(q)=rq,$ $\,\,\,$ $r \in
\mathbb{R}^{+}\backslash\{0\};$ 
\item[iv)] $L_4(q)=q^{-1}.$ 
\end{itemize}
\noindent Moreover, all the elements of $\mathbb{G}$ are conformal.
\end{lemma}
\begin{proof} Let us consider the fractional linear transformation 
    $L_{A}(q)=(aq+b)(cq+d)^{-1}$. If $c=0$ then $L_{A}(q)=
    (aq+b)d^{-1}=[d(aq+b)^{-1}]^{-1}$.
If instead $c\ne 0$, simply notice that
\begin{equation*}
L_{A}(q)=(aq+b)(cq+d)^{-1}=ac^{-1}+(b-ac^{-1}d)(cq+d)^{-1}
\end{equation*}
where $(b-ac^{-1}d)\ne 0$ since $det_{\mathbb{H}}(A)\ne 0$ (see
proposition \ref{uguaglianza}). This concludes the proof of the first
part of the statement. 
The proof of the conformality of all the elements of $\mathbb{G}$ can
be accomplished by observing that $L_1, L_2, L_3$ are obviously
conformal, and by proving that $L_4$ is conformal as well. In fact
the conjugation $q \mapsto \overline{q}$ is conformal and the
$\mathbb{R}-$differential of the map
$\overline{L_4(q)}=\frac{q}{|q|^2}$ at the point
$q=x_0+x_1i+x_2j+x_3k \equiv (x_0, x_1, x_2, x_3 )$ is represented
(up to multiplying by 
$\frac{1}{(x_0^2+x_1^2+x_2^2+x_3^2)^2}$)
by
the conformal matrix
$$
\left[ \begin{array}{cccc}
-x_0^2+x_1^2+x_2^2+x_3^2&  -2x_1x_0 &-2x_2x_0 & -2x_3x_0 \\
       -2x_0x_1       & x_0^2-x_1^2+x_2^2+x_3^2 &  -2x_2x_1      &
-2x_3x_1           \\
 -2x_0x_2& -2x_1x_2 &x_0^2+x_1^2-x_2^2+x_3^2& -2x_3x_2      \\
  -2x_0x_3        &  -2x_1x_3 &-2x_2x_3    &x_0^2+x_1^2+x_2^2-x_3^2\\
 \end{array} \right]
$$
\end{proof}

\begin{prop} \label{Birconj}
With reference to lemma \textnormal{\ref{struttura di G}}, if $L\in
\mathbb{G}$ is of type \textnormal{i)} or \textnormal{iii)}, then
\begin{equation*}
\ratio (L(q_1),L(q_2),L(q_3),L(q_4))=\ratio(q_1,q_2,q_3,q_4).
\end{equation*}
If instead $L\in \mathbb{G}$ is of type \textnormal{iv)}, then
\begin{equation*}
\ratio(L(q_1),L(q_2),L(q_3),L(q_4))=q_3  \ratio(q_1,q_2,q_3,q_4)
q_3^{-1}.
\end{equation*}
Finally, if $L(q)=aq$ is of type  \textnormal{ii)} , then
\begin{equation*}
\ratio(L(q_1),L(q_2),L(q_3),L(q_4))=a \ratio(q_1,q_2,q_3,q_4) a^{-1}
\end{equation*}
with $a \in \mathbb{S}^3_{\mathbb{H}}.$
\end{prop}
\noindent This last statement, whose proof is a straightforward
computation, has interesting consequences which will lead us to find
out peculiar geometric properties of the quaternionic fractional
linear transformations. Denote, as already established, by
$\mathbb{S}$ the $2-$sphere of pure imaginary units $\{x_1i+x_2j+x_3k
\in \mathbb{H} : x_1^2+x_2^2+x_3^2=1\}$ of $\mathbb{H}$ and consider,
for $x, y\in \mathbb{R}$, the $2-$sphere $x+y\mathbb{S}$ with center
$x$ and radius $|y|$. Then

\begin{lemma}\label{normale} For any $2-$sphere $x+y\mathbb{S}$  and
any $q\in \mathbb{H}\backslash \{0\}$, we have
$q(x+y\mathbb{S})q^{-1}=x+y\mathbb{S}$.
\end{lemma}
\begin{proof} For any $x+yI \in x+y\mathbb{S}$, we have
$q(x+yI)q^{-1}=qxq^{-1} + qyIq^{-1}=x+yqIq^{-1}$. Now $|qIq^{-1}|=1$
and $\Re e (qIq^{-1}) =\Re e (Iq^{-1}q) = \Re e (I) =0$. Therefore
$qIq^{-1}\in
\mathbb{S}$, which concludes the proof.
\end{proof}
Proposition \ref{Birconj} and lemma \ref{normale} directly imply that
the orbits of the cross-ratio of four quaternions (under the action
of the group of fractional linear transformations) are $2-$spheres of
type $x+y\mathbb{S}$, as established in the following 
\begin{corollario}\label{crreale}
Let $\ratio(q_1, q_2, q_3, q_4)=x+yI\in x+y\mathbb{S}$ be the
cross-ratio of the four quaternions $q_1, q_2, q_3, q_4$. Then
$\{\ratio (L(q_1),L(q_2),L(q_3),L(q_4)) : L \in \mathbb{G}\} =
x+y\mathbb{S}$. In particular, when the cross ratio of four
quaternions is real, then it is invariant under the action of all
fractional, linear transformations.
\end{corollario}

Let us set $\mathcal{S}_3 =\{ q+r\mathbb{S}_\mathbb{H}^3 : q \in
\mathbb{H}, r \in \mathbb{R}^+ \setminus \{ 0 \} \}$ to be the family
of all $3-$(real)-dimensional spheres of $\mathbb{H},$ and
denote by $\mathcal{P}_3$ the family of all $3-$(real)-dimensional
affine spaces of $\mathbb{H}.$ If $\mathcal{F}_3= \mathcal{S}_3 \cup
\mathcal{P}_3,$ then we can state the following result, which closely
resembles the classical statement that holds for all fractional
linear transformations of $\mathbb{C}.$ 
\begin{prop}\label{sfretdim3}
The group $\mathbb{G}$ of all fractional, linear transformations maps
elements of $\mathcal{F}_3$ onto elements of $\mathcal{F}_3$, i.e.
it transforms the family of all $3-$spheres and $3-$dimensional,
affine planes of $\mathbb{H}$ onto itself.
\end{prop}
\begin{proof}
Indeed the family of sets $\mathcal{F}_3$ is the family of zero-sets
of the quadratic equations
\begin{equation} \label{trisfret}
\alpha (q \overline{q})+\beta q + \overline{q} \overline{\beta} +
\gamma = 0 
\end{equation}
where $\alpha, \gamma \in \mathbb{R}$ and $\beta \in \mathbb{H}.$ In
fact, if we set $q=x_0+x_1i+x_2j+x_3k$ and $\beta=\beta_0 + \beta_1 i
+\beta_2 j + \beta_3 k,$ equation (\ref{trisfret}) becomes
$$
\alpha (x_0^2+x_1^2+x_2^2+x_3^2)+ 2\Re e (\beta q) + \gamma = 0
$$
i.e.
\begin{equation}\label{esplicita}
\alpha (x_0^2+x_1^2+x_2^2+x_3^2)+ 2(\beta_0x_0-\beta_1x_1-\beta_2x_2
-\beta_3x_3) + \gamma = 0.
\end{equation}
By varying $\alpha, \gamma$ in $\mathbb{R}$ and $ \beta$ in
$\mathbb{H}$, we obtain the entire family  $\mathcal{P}_3$ as the
family of zero-sets of (\ref{esplicita}) when $\alpha = 0$, and the
entire family $\mathcal{S}_3$ when $\alpha \neq 0$.  
At this point,  it is enough to prove that the elements of
$\mathbb{G}$ transform an equation of type (\ref{trisfret}) in an
equation of the same type. If $L_1(q)=q+b,$ with $b\in \mathbb{H}$,
then equation (\ref{trisfret}) becomes
\begin{eqnarray}
\alpha( (q+b)\overline{(q+b)})+\beta (q+b) + \overline{(q+b)}\,\,
\overline{\beta} + \gamma = 0 \\ \nonumber
\alpha (q \overline{q}) + \alpha (2 \Re e (q \overline{b})) + 2
\Re e (\beta q) +\alpha |b|^2 + 2 \Re e (\beta b) +\gamma =0\\
\nonumber
\alpha (q \overline{q}) + \alpha (2 \Re e (\overline{b} q)) + 2
\Re e(\beta q) +\alpha |b|^2 + 2 \Re e (\beta b) +\gamma=0 \\
\nonumber
\alpha (q \overline{q}) + 2 \Re e ((\alpha \overline{b} +\beta)q)
+\alpha |b|^2 + 2 \Re e (\beta b) +\gamma=0 \\\nonumber
\alpha (q \overline{q}) + (\alpha \overline{b}
+\beta)q+\overline{q}\overline{(\alpha \overline{b} +\beta)}
+\alpha |b|^2 + 2 \Re e (\beta b) +\gamma=0 \nonumber
\end{eqnarray}
which is still an equation of the same type. If  $L_2(q)=aq,$ with $a
\in \mathbb{S}^3_\mathbb{H}$, then (\ref{trisfret}) becomes
\begin{eqnarray}
\alpha (aq)( \overline{aq})+ \beta (aq) + (\overline{aq})
\overline{\beta} + \gamma =0\\ \nonumber
\alpha (q \overline{q}) + (\beta a)q + \overline{q} (\overline{a}
\overline{\beta}) + \gamma =0 \\ \nonumber
\alpha (q \overline{q}) + (\beta a)q + \overline{q} (\overline{\beta
a}) + \gamma =0. \nonumber
\end{eqnarray}
Under the action of $L_3(q)=rq$, with $r\in
\mathbb{R}^+\backslash\{0\}$, equation (\ref{trisfret}) transforms
into
\begin{equation}
r^2\alpha (q \overline{q}) + (r\beta) q
+\overline{q}(\overline{r\beta})+ \gamma =0
\end{equation}
while, for $L_4(q)=q^{-1}$, it becomes
\begin{eqnarray}
\alpha + \beta \overline{q} + q \overline{\beta} + \gamma (q
\overline{q}) =0\\ \nonumber
|\beta|^2\alpha + |\beta|^2 \overline{q}\beta + \overline{\beta}q
|\beta|^2 + |\beta|^2\gamma (q \overline{q}) =0\\ \nonumber
\gamma (q \overline{q}) + \overline{\beta}q +\overline{q}\beta
+\alpha =0.\nonumber
\end{eqnarray}
What is established in lemma \ref{struttura di G} leads to the
conclusion of the proof.
\end{proof}

The geometrical properties of the elements of the group $\mathbb{G}$
are quite interesting: they are a generalization, and an extension to
higher dimensions, of the geometrical properties of the classical
group of complex fractional linear transformations. To give a clear
idea of what we mean by this, we will denote by $\mathcal{F}_i,$ for
$i=1,2,$
the family of all $i-$(real)-dimensional spheres and $i-$(real)-dimensional
affine spaces of $\mathbb{H}$ and state the following:
\begin{corollario} \label{cfreret}
The group $\mathbb{G}$ of all fractional, linear transformations maps
elements of $\mathcal{F}_2$ onto elements of  $\mathcal{F}_2$ and
elements of $\mathcal{F}_1$ onto elements of  $\mathcal{F}_1$, 
i.e. it transforms the family of all $2-$spheres and $2-$dimensional, 
affine planes of $\mathbb{H}$ onto itself and the family of all
circles and affine lines of $\mathbb{H}$ onto itself.
\end{corollario}
\begin{proof}
Since all the elements of $\mathcal{F}_2$ and $\mathcal{F}_1$ are
obtained as finite intersections of element of $\mathcal{F}_3$, the
proof is a consequence of proposition \ref{sfretdim3}.
\end{proof}

The above Corollary  will play a key role while, in the sequel, we
will define the Poincar\'e distance on the open unit disc
$\Delta_\mathbb{H}$ of $\mathbb{H}$. To prepare the tools to be able
to give such a definition, we will study the characterizing
properties of the quaternionic cross-ratio.

\begin{teo}\label{birreal}
Four pairwise distinct points $q_1,$$q_2,$$q_3,$$q_4 \in \mathbb{H}$
lie on a same
(one-dimensional) circle if, and only if, their cross-ratio
is real. The two pairs of points $q_1,q_2$ and $q_3, q_4$ lying on a
same
circle separate each other if, and only if,
$\ratio(q_1,q_2,q_3,q_4)<0.$ 
\end{teo}
\begin{proof1}
The three pairwise distinct points $q_2, q_3, q_4$ determine a unique
circle (or line) $C \subset \mathbb{H}$.
In view of proposition \ref{0 1 infinito}, take $L\in \mathbb{G}$
that maps $q_2, q_3, q_4$ respectively
to $1,0,\infty$ and let $q_0=L(q_1).$ 
Then, by Corollary \ref{cfreret}, $L$ carries $C$ onto the real axis
$\mathbb{R}$ of $\mathbb{H}$ and,
by proposition \ref{Birconj}, it is such that
$\ratio(L(q_1),L(q_2),L(q_3),L(q_4))= \ratio(q_0,1,0 , \infty)=q_0$
is 
conjugated to $\ratio(q_1, q_2, q_3, q_4)$. We conclude that $q_0\in
\mathbb{R}$ if, and only if, $q_1\in C$. 
Equivalently $\ratio(L(q_1),L(q_2),L(q_3),L(q_4))= q_0 \in
\mathbb{R}$ if, and only if, $\ratio(q_1, q_2, q_3, q_4)\in
\mathbb{R}$ if, and only if, $q_1,$$q_2,$$q_3,$$q_4 \in C$.
This proves the first part of our assertion.  

\noindent To complete the proof notice that
$\ratio(q_0,1,0,\infty)=q_0<0$ if, and only if, the two pairs of
points $q_0,1$ and $0,\infty$ separate each other on the real axis.
Since $L$ maps the circular arc $A$ from
$q_3$ to $q_4$ through $q_2$ onto the positive real half axis,
then (by the continuity of $L$) the pre-image $q_1$ of $q_0$ cannot
belong to $A$. Therefore $q_1,q_2$ and $q_3, q_4$ separate
each other.
\end{proof1}
When defining the Poincar\'e distance on the open, unit disc of
$\mathbb{H}$, we will be interested in the case in which the two
pairs of points $q_1,q_2$ and $q_3,q_4$ lie on a same circle and do
not separate
each other. In this case, if we keep $q_2,q_3,q_4$ fixed and move
$q_1$ from
$q_3$ to $q_4,$ by way of $q_2,$ then,  within the environment
established in the proof of theorem \ref{birreal},  the point $q_0$
moves from $0$ to
$\infty$ by way of $1$. In view of 
\begin{equation} \label{birapp}
\ratio(q_1,q_2,q_3,q_4)=\ratio(q_0,1,0,\infty)=q_0
\end{equation}
during this procedure the cross ratio $\ratio(q_1,q_2,q_3,q_4)$ takes
on all
the positive values; in particular the value $1$ comes up when
$q_1=q_2.$ Moreover

\begin{prop} \label{arrcyc}
Let $q_1,q_2,q_3,q_4 \in \mathbb{H}$ be pairwise distinct points, 
arranged cyclically on a circle. Then 
$\ratio(q_1,q_2,q_3,q_4)>1$. Moreover, $q_{1}=q_{2}$ if, and only
 if,
$\ratio(q_1,q_2,q_3,q_4)=1.$
\end{prop}
\begin{proof1}
If $q_1,q_2,q_3,q_4$ are arranged cyclically, so are
$q_0,1,0,\infty$. Therefore, in view of (\ref{birapp}),
 we have $q_0 = \ratio(q_0,1,0,\infty)>1.$
\end{proof1}

\section{M\"obius transformations and the Poincar\'e distance on
$\Delta_{\hh}$}

We are now ready to construct a Poincar\'e-type distance, which we
will simply 
call Poincar\'e distance, on the open unit disc $\Delta_\mathbb{H}$
of $\mathbb{H}$. We will do this by developing, in the quaternionic
case, a variation
of an approach adopted by Ahlfors in a different algebraic situation
placed in the
$n-$dimensional real vector space $\mathbb{R}^{n}$, \cite{Ah3}.

\noindent We will start by defining the  non-euclidean line through
any two points
$q_1,q_2 \in \Delta_{\mathbb{H}}$. We will use a ``slicewise''
approach.

\begin{definiz}\label{noneuclidean} If $q_1\neq q_2\in
\Delta_{\mathbb{H}}$ are $\mathbb{R}-$linearly 
dependent, i.e. if they
lie on a same diameter of the disc
$\Delta_{\mathbb{H}} \subset \mathbb{H} \cong \mathbb{R}^4,$
then we define the {\it non-Euclidean line} through $q_1$
and $q_2$ to be this diameter. When $q_1,q_2$ are
$\mathbb{R}-$linearly 
independent, then they belong to a unique circle that intersects 
$\mathbb{S}^{3}_\mathbb{H}=\partial \Delta_{\mathbb{H}}$ orthogonally
and that will be
defined to be the {\it non-Euclidean line} through $q_1$
and $q_2$.
\end{definiz}

To clarify the geometrical significance of the above definition, let
us remark that any circle $C$ which intersects 
$\mathbb{S}^{3}_{\mathbb{H}}$ orthogonally 
belongs to the $2-$dimensional, real vector space $\Pi(C)$ 
spanned by the two vectors
obtained as $\mathbb{S}^{3}_\mathbb{H}\cap C$. Now, when $q_1, q_2\in
C$ are
$\mathbb{R}-$linearly independent, they span a
$2-$dimensional, real vector space $\Pi(q_{1}, q_{2})\subset
\mathbb{H}$, which obviously must coincide with $\Pi(C)$. Therefore
$C$ is the classical non-Euclidean line of the
$2-$(real)-dimensional, open,
unit disc
$\Delta_{\mathbb{H}}\cap \Pi(q_{1}, q_{2})$ passing through $q_1,
q_2$.

\begin{teo}\label{quattropunti}
For any given $q_1,q_2 \in  \Delta_{\mathbb{H}}$,  with $q_1 \neq
q_2,$ the unique 
non-Euclidean line $l$ containing $q_1$ and $q_2$ is
the circle or the straight line determined by the four points
$q_1,q_2,\overline{q_1}^{-1},\overline{q_2}^{-1}.$
\end{teo} 
\begin{proof1}
Suppose $q_1 ,q_2$ are $\mathbb{R}-$linearly independent. Since
$\overline{q_1}^{-1} =q_1|q_1|^{-2}$ and
$\overline{q_2}^{-1}=q_2|q_2|^{-2}$,
the four given points determine the $2-$dimensional real subspace
$\Pi(q_{1}, q_{2})$
spanned by $q_1, q_2$.
An easy computation shows that
\begin{equation*}
\ratio(q_1,q_2,\overline{q_1}^{-1},\overline{q_2}^{-1})=|q_1|^2|q_2|^2-q_2\overline{q_1}-q_1\overline{q_2}
+ 1 \in \mathbb{R}.
\end{equation*} 
Thus $q_1,q_2,\overline{q_1}^{-1},\overline{q_2}^{-1}$ lie on a same
circle $l\subset \Pi(q_{1}, q_{2}).$ Finally, to prove that the circle
$l$ is orthogonal to $\mathbb{S}^3_{\mathbb{H}} ,$ we notice that the
points 
$q_1,q_2, q_1|q_1|^{-2}, q_2|q_2|^{-2}\in \Pi(q_{1}, q_{2})\equiv
\mathbb{R}^2$, when placed on the complex plane $\mathbb{C}$ via the
identification $\mathbb{R}^2 \cong \mathbb{C},$ can still be written
as
$q_1,q_2,\overline{q_1}^{-1},\overline{q_2}^{-1}$.
Therefore the proof reduces to the classical proof for the complex
plane (see, e.g., \cite{Siegel}).  The remaining case in which $q_1,
q_2$ are $\mathbb{R}-$linearly dependent is straightforward.
\end{proof1}  

We now turn our attention to investigate the structure of the group
$\mathbb{M}$ of {\it M\"{o}bius transformations},
i.e. of the subgroup $\mathbb{M}$ of $\mathbb{G}$ consisting of all
fractional linear transformations 
 mapping the quaternionic, open, unit disc $\Delta_\mathbb{H}$ onto
itself. First of all we recall that, once
	     named $H=\left[ \begin{array}{rr}

             1 & 0 \\
             0 & -1 \\
             \end{array} \right]$, the (classical) group of matrices 
(with quaternionic entries)
$Sp(1, 1)$ is defined as (see, e.g., \cite{GOV})
\begin{equation}\label{sp11}
    Sp(1,1)=\left\{  A\in M(2, \mathbb{H})\  :\ \  ^{t}\overline{A}HA=H 
    \right\}
\end{equation}
and it can be written equivalently as (see, e.g., \cite{CPW})
\begin{equation*}
    Sp(1,1)=\left\{  \left[ \begin{array}{ll}

             a & b \\
             c & d \\
             \end{array} \right] : |a|=|d|, \ \ |b|=|c|, \ \  |a|^2 -|c|^2 =1, 
	     \ \ 
\overline{a}b=\overline{c}d, \ \ a\overline{c}=b\overline{d}
    \right\}.
\end{equation*}
In terms of $Sp(1,1)$, we can rephrase and complete a result of \cite{CPW} as
follows:
\begin{teo} \label{discmoeb} 
The quaternionic, fractional linear transformation defined by $g(q)=(aq+b)(cq+d)^{-1}$  is a M\"obius 
transformation of $ \Delta_{\mathbb{H}}$ if and only if $\left[ \begin{array}{ll}

             a & b \\
             c & d \\
             \end{array} \right] \in Sp(1,1)$.
  Moreover the map 
\begin{eqnarray}
   &\phi: Sp(1,1) \to \mathbb{M} \nonumber \\
   &A=\left[ \begin{array}{ll}
             a & b \\
             c & d \\
             \end{array} \right]\mapsto
L_{A}(q)=(aq+b)\cdot(cq+d)^{-1}
    \end{eqnarray}
    is a group homomorphism 
    whose kernel is the center of $Sp(1,1)$, that is the
    subgroup 
    $$\left\{\pm \left[ \begin{array}{ll}

             1& 0 \\
             0 & 1\\
             \end{array} \right] 
\right\}.
	     $$
\end{teo}
\noindent Notice, in particular, that for all  $A\in Sp(1,1),$ we have 
\begin{equation}
det_{\mathbb{H}}(A)=\sqrt{|a|^{4}+|c|^{4}-2|a|^{2}|c|^{2}}=(|a|^{2}-|c|^{2})=1
\end{equation}	       
and hence $Sp(1,1)\subset SL(2, \mathbb{H})$.
\vskip 0.2cm
	     By means of the statement of theorem \ref{discmoeb} we are
able to obtain, for the
quaternionic 
M\"obius transformations, a characterization  which closely resembles
the classical
representation of the complex M\"obius transformations. A similar
result is stated without proof in \cite{HJ}.

\begin{teo}\label{teofonda}
Each quaternionic M\"{o}bius transformation
$g(q)=(aq+b)\cdot(cq+d)^{-1}\in \mathbb{M}$
can be
written uniquely as:
\begin{equation} \label{formafond}
g(q)=\alpha (q - q_0)(1-\overline{q_0}q)^{-1}  \beta^{-1}
\end{equation}
where $q_0 = -u \tanh(t) = -a^{-1} b \in \Delta_{\mathbb{H}}$ and
where
$\alpha =\dfrac{a}{|a|}\in \mathbb{S}^{3}_{\mathbb{H}}$, 
$\beta = \dfrac{d}{|d|} \in \mathbb{S}^3_{\mathbb{H}}.$ 
\end{teo}
\begin{proof1} \\
By $|a|^2-|c|^2=1$ and by $|d|^2-|b|^2=1$ stated in Proposition
\ref{discmoeb}, we obtain:
$$ a= \alpha \cosh(t) \,\,\, , \,\,\, b=\gamma \sinh(t)$$
$$ d=\beta \cosh(t)    \,\,\, , \,\,\, c= \delta \sinh(t)$$
with $\alpha, \beta, \gamma, \delta \in \mathbb{S}^3_{\mathbb{H}}.$
By $\overline{a}b=\overline{c}d,$ we have:
$$\overline{\alpha}\gamma =\overline{\delta}\beta. $$
We recall that $\overline{\alpha}=\alpha^{-1}$ because $|\alpha|=1,$
hence:
$$ \alpha^{-1} \gamma = \delta ^{-1} \beta .$$ 
If $u:=\alpha^{-1} \gamma=\delta^{-1} \beta,$ then
$$\alpha = \gamma u^{-1}$$
$$\beta =\delta u.$$
Finally 
$$G= \left[ \begin{array}{ll}

             \gamma u^{-1} \cosh(t) & \gamma \sinh(t) \\
             \delta \sinh(t) & \delta u \cosh(t) \\
             \end{array} \right] $$
and the fractional linear map associated to $G$ becomes:
$$ g(q) = ((\gamma u^{-1} \cosh(t))q + \gamma \sinh(t))(\delta
\sinh(t))q +\delta u \cosh(t))^{-1}.$$
We extract $\gamma u^{-1} \cosh(t)$ from the first factor and $\delta
u \cosh(t)$ from the second factor and we obtain:
$$g(q)= (\gamma u^{-1} \cosh(t))( q + u \tanh(t))(u^{-1}
\tanh(t) q + 1)^{-1}(\delta u \cosh(t))^{-1}= $$
$$=(\gamma u^{-1})( q + u \tanh(t))(u^{-1} \tanh(t) q +
1)^{-1} u^{-1} \delta^{-1}$$
where $\gamma,$ $u,$ $\delta$ $\in \mathbb{S}^3_{\mathbb{H}}$ and $u
\tanh(t)
=\overline{u^{-1} \tanh(t)}.$
Therefore the M\"{o}bius transformations of the unit disc are of the
form:
$$g(q)= \alpha ( q - q_0)(1- \overline{q_0}q)^{-1}
\beta ^{-1}$$
where $\alpha, \beta \in \mathbb{S}^3_{\mathbb{H}}$ and where 
$q_0 = -u \tanh(t) = -a^{-1} b \in \Delta_{\mathbb{H}}$.
It is now an easy exercise to verify that the maps of the form
(\ref{formafond}) 
transform the unit disc of $\mathbb{H}$ onto itself. 
Indeed:
$$ 1-\alpha (q - q_0)(1-(\overline{q_0})q)^{-1}\beta^{-1}
\overline{\beta^{-1}} 
(1-\overline{q} q_0)^{-1} 
(\overline{q} - \overline{q_0}) \overline{\alpha}=$$
$$= 1 - |(q-q_0)|^2 |(1-\overline{q_0}q)^{-1}|^2=$$
$$=(|1- \overline{q_0}q|^2 - |q - q_0|^2)
|(1-\overline{q_0}q)^{-1}|^2=$$
$$=(1-\overline{q}q_{0}-\overline{q_{0}}q+|q_{0}|^{2}|q|^{2}-|q|^{2}+q\overline{q_{0}}+
q_{0}\overline{q}-|q_{0}|^{2}) |(1-\overline{q_0}q)^{-1}|^2=$$
$$=(1-q_0 \overline{q_0})(1-q \overline{q}) 
|(1-\overline{q_0}q)^{-1}|^2$$
because $2 \Re e (\overline{q} q_0) - 2 \Re e (q \overline{q_0})=0.$
\vskip 0.5cm
An alternative proof can be obtained directly as follows. 
If $g(q)=(aq+b) (cq+d)^{-1}\in \mathbb{G}$ belongs to the group
of the M\"obius transformations $\mathbb{M}$ and fixes
$0,$ then $b=0$  and, by the given characterization of $Sp(1,1)$, 
$c=0$ and $|a|=|d|=1$. 
Therefore each M\"{o}bius transformation which fixes
0 is of type $g(q)=a q  d^{-1}$.
Now let $g$ be a M\"{o}bius transformation such that
$g(0)=-p_{0}=bd^{-1}.$ 
If we compose $g$ with $h(q)=(q +p_{0}) (1+\overline{p_{0}} q)^{-1},$ 
then $(h \circ g)$ fixes $0,$ and hence $(h
\circ g)(q)=a q d^{-1}$. Finally
$$
g(q)=h^{-1} (a  q d^{-1})=(a q  d^{-1} -
p_{0})(1 - \overline{p_{0}}  a q 
d^{-1})^{-1}
$$
$$
=a( q - \overline{a} p_{0}d) \overline{d} d (1 - \overline{d}
\overline{p_{0}}a  q)^{-1} d^{-1} 
$$
$$
= a(q-q_{0})(1-\overline{q_0} q)^{-1} d^{-1}$$
where $q_0= \overline{a} p_{0}d=a(-bd^{-1})d=-a^{-1}b.$
\end{proof1}

As we already mentioned, the M\"obius transformations form a subgroup
of the group $\mathbb{G}$ of all fractional linear transformations of
$\mathbb{H}$. It is of interest to consider two M\"obius
transformations in their form (\ref{formafond}) and find the form
(\ref{formafond}) of their composition. Indeed, given the two
transformations
\begin{equation} \label{formafonda}
g_1(q)=a(q - q_0)(1-\overline{q_0}q)^{-1} b^{-1}
\end{equation}
and
\begin{equation} \label{formafondb}
g_2(q)= c(q - p_0)(1-\overline{p_0}q)^{-1} d^{-1}
\end{equation}
(with $|a|=|b|=|c|=|d|=1$ and $|q_0|<1$, $|p_0|<1$) it is easy to
verify that the composition $g= g_1\circ g_2$ is the transformation
associated to the matrix
$$ \left[ \begin{array}{ll}

             a & -aq_0 \\
           - b\overline{q_0} & b \\
             \end{array} \right]
             \left[ \begin{array}{ll}
              c & -cp_0 \\
             -d\overline{p_0} & d \\
             \end{array} \right] =
             \left[ \begin{array}{ll}
             ac + a q_0 d \overline{p_0} & -(ac p_0 + a q_0 d )  \\
             -(b \overline{q_0} c + bd \overline{p_0}) & b
\overline{q_0} c p_0 + bd \\
             \end{array} \right].
$$
Since
$$
|ac+a q_0 d \overline{p_0}|=|ac||1 +c^{-1}q_0 d \overline{p_0}|= [1 +
|p_0|^2 |q_0|^2 + 2 \Re e (p_0\overline{d}\overline{q_0} c)]^{1/2}
$$ 
$$
|b \overline{q_0} c p_0 + bd|=|bd||1+ d^{-1}\overline{q_0} c p_0|= [1
+|p_0|^2|q_0|^2 + 2 \Re e (\overline{p_0} \,\,\, \overline{c} q_0
d)]^{1/2}
$$
and since
$$
\Re e (p_0\overline{d}\overline{q_0} c)=\Re e
(cp_0\overline{d}\overline{q_0} )=\Re e
(\overline{(cp_0)}\overline{(\overline{d}\overline{q_0} )})=\Re e
(\overline{p_0} \,\,\, \overline{c} q_0 d)
$$
then
$$
|ac+a q_0 d \overline{p_0}|=|b \overline{q_0} c p_0 + bd|.
$$
Moreover
\begin{eqnarray*}
\overline{(ac+a q_0 d \overline{p_0})}( acp_0 + aq_0d)
=(p_0 \overline{d} \overline{q_0}\, \overline{a} + \overline{c}\,
\overline{a})( acp_0 + aq_0d)\\
=p_0 + p_0|q_0|^2 + \overline{c}q_0 d + p_0 \overline{d}
\overline{q_0} c p_0
\end{eqnarray*}
and
\begin{eqnarray*}
\overline{(b \overline{q_0} c p_0 + bd)} (b \overline{q_0}c + bd
\overline{p_0})
=(\overline{d}\,\overline{b}+ \overline{p_0}\,\overline{c} q_0
\overline{b}) (b \overline{q_0}c + bd \overline{p_0})\\
=\overline{p_0} + 
\overline{p_0}|q_0|^2 + \overline{d} \overline{q_0} c +
\overline{p_0} \,\,\, \overline{c} q_0 d \overline{p_0}.
\end{eqnarray*}
In conclusion we can write
\begin{equation} \label{formafondcomp}
g(q)=\alpha (q - w_0)(1-\overline{w_0}q)^{-1} \beta^{-1}
\end{equation}
where 
$$
\alpha =  \frac{(ac + a q_0 d \overline{p_0})}{|ac + a q_0 d
\overline{p_0}|},
\hskip 1.5cm
\beta=  \frac{bd+b \overline{q_0} c p_0}{ |bd+b \overline{q_0} c
p_0|},
$$
$$
w_0=(ac + a q_0 d \overline{p_0} )^{-1}(ac p_0 + a q_0 d)=\frac{p_0 +
p_0|q_0|^2 + \overline{c}q_0 d + p_0 \overline{d} \overline{q_0} c
p_0}{|ac + a q_0 d \overline{p_0}|^2}.
$$
Notice that the inverse of the M\"obius transformation 
$g(q)=\alpha (q - q_0)(1-\overline{q_0}q)^{-1}  \beta^{-1}$ is given
by $g^{-1}(q)=\alpha^{-1}(q + \alpha
q_0\overline{\beta})(1+\beta\overline{q_0}\ \overline{\alpha}q)^{-1}\beta
$.

\begin{nota} The determinant $det_{\mathbb{H}}(M)$ of the matrix $M$
associated 
to the M\"obius transformation
$g(q)=\alpha  (q - q_0)(1-\overline{q_0}q)^{-1}  \beta^{-1}$
is equal to $(1-|q_{0}|^{2})$, in accordance with what happens in
the
complex case.
\end{nota}

We are now ready to develop the announced geometric approach to the definition of the 
quaternionic Poincar\'e  distance on 
$\Delta_{\mathbb{H}}$. To this end, we consider the non-euclidean
line 
$l$ determined by
$q_1, q_2 \in \Delta_{\mathbb{H}}$ as given in definition
\ref{noneuclidean} and call {\it ends} of $l$ the two intersection
points $l\cap \mathbb{S}^{3}_{\mathbb{H}}$.
We name $q_3$ and $q_4$ such ends, so that
$q_1,$$q_2,$$q_3,$$q_4$ are arranged cyclically on $l$.
Then we define:
\begin{equation}\label{definizionedistanza}
\delta_{\Delta}(q_1,q_2)=\frac{1}{2}\log (\ratio(q_1,q_2,q_3,q_4))
\end{equation}
to be the {\it Poincar\'e distance} between $q_1$ and $q_2.$ 

\begin{prop}\label{distanza} The map $\delta_{\Delta}(q_1,q_2)=\frac{1}{2}\log
    (\ratio(q_1,q_2,q_3,q_4))$ defined in (\ref{definizionedistanza}) 
    is a distance.
\end{prop}

Before
proving that $\delta_{\Delta}$ is actually a distance, we need to state
properties of invariance for it. The following geometrical feature of
the elements of
$\mathbb{M}$ is important in the sequel.

\begin{lemma}\label{lines onto lines}
The M\"obius transformations map non-Euclidean lines of
$\Delta_{\mathbb{H}}$ onto
non-Euclidean lines of $\Delta_{\mathbb{H}}$.
\end{lemma} 
\begin{proof} Thanks to Lemma \ref{struttura di G}, the M\"obius
transformations are conformal. 
Since any M\"obius transformation maps $\mathbb{S}^3_\mathbb{H}$ onto
itself, the proof is concluded in view of  corollary
\ref{cfreret}.
\end{proof}

We will now make some remarks on the geometrical
properties of the M\"obius transformations. For a non real $q_{0}$,
let us consider 
\begin{equation*}
  g(q)= (q - q_0)(1-\overline{q_0}q)^{-1}.
\end{equation*} 

\begin{nota}\label{geometria moebius}
The map $g$
transforms any $2-$dimensional real plane containing the line
$l_{q_{0}}=\{ tq_0:  t \in \mathbb{R}\}$ onto  a  $2-$dimensional
real plane containing the same line
$l_{q_{0}}$. Any other 
$2-$dimensional real vector subspace not containing the point $q_{0}$ is
mapped 
onto a $2-$dimensional sphere orthogonal to
$\mathbb{S}^{3}_\mathbb{H}$. 
\end{nota}
\noindent This last remark can be explained  as follows:  since $g(q_0)=0$ and
$g(0)=-q_0$, by corollary \ref{cfreret} and in view of the
conformality of $g$ (see lemma \ref{struttura di G}), we get that
$g(l_{q_0})=l_{q_0}$. Therefore corollary \ref{cfreret} leads to the
proof of the first part of the remark. If a $2-$dimensional real
plane $\pi$ does not contain the line $l_{q_{0}}$, then $0 \notin
g(\pi)$ and, again by conformality and by corollary \ref{cfreret}, we
obtain that $g(\pi)$ is a $2-$dimensional sphere orthogonal to
$\mathbb{S}^{3}_\mathbb{H}$.

Let us now consider the group $\mathbb{M}^{*}$ of {\it extended M\"obius
transformations} defined as the union of all the M\"obius
transformations $g \in \mathbb{M}$ and all maps $h$ obtained as
$h(q)=g(\overline{q})$ for $g\in \mathbb{M}$.

\begin{prop}\label{invarianza} The Poincar\'e distance of
    $\Delta_{\mathbb{H}}$ is invariant under the action of the group
    of all extended M\"obius transformations $\mathbb{M}^{*}$. 
\end{prop}
\begin{proof} Let us start the proof by recalling that the map
$q\mapsto 
    \overline{q}$ is conformal and transforms $\Delta_{\mathbb{H}}$
onto
    itself. Therefore, as all M\"obius transformations, it maps
non-Euclidean lines onto non-Euclidean
    lines transforming ends in ends. To conclude the proof, let us
observe that the cross-ratio, when real, is invariant with respect to
    the action of all elements of $\mathbb{M}$ (see corollary
    \ref{crreale}) and with respect to $q\mapsto \overline{q}.$
\end{proof}

\begin{proof2} By proposition \ref{arrcyc}, the cross ratio
$\ratio(q_1,q_2,q_3,q_4)$ 
of the four points is a
real positive number greater than $1$ and therefore its real
logarithm is well
defined and 
positive. We want to prove now that $\delta_{\Delta}$ is symmetric.
Interchanging $q_1$ and $q_2$ 
requires interchanging $q_3$ and $q_4$
to maintain the cyclical order. After simple computations, we find:
\begin{equation*}
\ratio(q_1,q_2,q_3,q_4)=k_1  k_2 \in \mathbb{R}
\end{equation*}
where $k_1=(q_1-q_3)(q_1-q_4)^{-1} \in \mathbb{H}$ and $k_2=
(q_2-q_4)(q_2-q_3)^{-1} \in \mathbb{H}.$
Similarly:
\begin{equation*}
\ratio(q_2,q_1,q_4,q_3)=k_2  k_1 \in \mathbb{R}.
\end{equation*}
Since $k_1 k_2 \in \mathbb{R}$ it follows that $k_1  k_2
=k_2  k_1.$ Hence
\begin{equation*}
\ratio(q_1,q_2,q_3,q_4)=\ratio(q_2,q_1,q_4,q_3)
\end{equation*}
and hence 
\begin{equation*}
\delta_{\Delta}(q_1,q_2)=\delta_{\Delta}(q_2,q_1)     
\end{equation*}
i.e. $\delta_{\Delta}$ is symmetric. We also have that $q_1=q_2$ if, and only 
if,  $\delta_{\Delta}(q_1,q_2)=0$ because
$\ratio(q_1,q_1,q_3,q_4)=1$ if, and only if, $q_1=q_2$ (see
proposition
\ref{arrcyc}). The last thing to prove is the triangle inequality. To 
this purpose, for any $q_{0}, q_{1}, q_{2}\in \Delta_{\mathbb{H}}$ we
have to show that
\begin{equation}\label{triangolo}
\delta_{\Delta}(q_{1}, q_{2})\le \delta_{\Delta}(q_{1}, q_{0}) + \delta_{\Delta}(q_{0}, q_{2}).
\end{equation}
In view of Theorem \ref{teofonda}, consider the following M\"{o}bius
transformation of
$\Delta_{\mathbb{H}}$ onto itself, 
\begin{equation} \label{realpoints}
L(q)=\lambda_1 (q-a)(1-\overline{a}q)^{-1} \lambda_2  
\end{equation}
where $a=q_1,$ $\lambda_1=|q_2-q_1|(q_2 -q_1)^{-1},$
$\lambda_2=(1-\overline{q_1}q_2)|1-\overline{q_1}q_2|^{-1}$ and
$t=|q_2-q_1||1-\overline{q_1}q_2| \in \mathbb{R}.$
The transformation $L$ maps $q_1$ to $0$ and $q_2$ to $t \in
\mathbb{R}^+.$  By proposition \ref{invarianza},  the map $\delta_{\Delta}$ is
invariant under 
the action of $L$, and hence (\ref{triangolo}) is equivalent to
\begin{equation}\label{triangolo1}
\delta_{\Delta}(0, t)\le \delta_{\Delta}(0, L(q_{0})) + \delta_{\Delta}(L(q_{0}), t).
\end{equation}
Let $L(q_0)=x+yI,$ for some $x,y \in \mathbb{R},$ $I \in \mathbb{S}.$
Clearly $L(q_0),0,t \in L_I \cap 
\Delta_{\mathbb{H}}=\Delta_I.$ Since, by construction,
$\delta_{\Delta}$ restricted to $\Delta_I$ coincides with the Poincar\'e distance of
$\Delta_I,$
we prove (\ref{triangolo1}) and conclude.
 \end{proof2}
 
At this point we are ready to exhibit a formula for the Poincar\'e
distance of $\Delta_{\hh}$. In fact by definition of $\delta_{\Delta}$ (see
(\ref{definizionedistanza}) )
we have, for any $t\in \rr^{+}$:
\begin{equation}\label{zerot}
    \delta_{\Delta}(0, t)=\frac{1}{2}\log (\ratio(0,t,1,-1))= \frac{1}{2}\log
    \frac{1+t}{1-t}.
\end{equation}
Since the Poincar\'e distance is invariant by rotations we get
\begin{equation}
\delta_{\Delta}(0,q)=\delta_{\Delta}(0,|q|)=
\frac{1}{2}\log\frac{1+|q|}{1-|q|}.
\end{equation}
In general, for arbitrary $q_1,q_2 \in \Delta_{\mathbb{H}},$ if we
consider the isometry
$q \mapsto (q-q_1)(1-\overline{q_1}q)^{-1},$ we obtain that:
\begin{equation} \label{formulalog}
\delta_{\Delta} (q_1,q_2)= \delta_{\Delta}
(0,|q_2-q_1||1-\overline{q_1}q_2|^{-1})=\frac{1}{2}\log\left( 
\frac{1+
|q_1-q_2||1-\overline{q_1}q_2|^{-1}}{1-|q_1-q_2||1-\overline{q_1}q_2|^{-1}}
\right).
\end{equation}

We will end this section by defining and studying the main
properties of  the analogous of the Poincar\'e differential 
metric - which we will often simply call {\it Poincar\'e metric} - in 
the case of the unit disc $\Delta_{\hh}$ of $\hh$.
To this aim we
mimic the definition of the classical complex
Poincar\'e differential metric of $\mathbb{D} \subset \mathbb{C}$ 
to set the length of the vector $\tau
\in \hh$  for the Poincar\'e metric at $q\in \Delta_{\hh}$ to be the
number:
\begin{equation}\label{poinmetric}
\langle \tau \rangle_{q} = \frac{|\tau|}{1-|q|^{2}}.
\end{equation}
Formula (\ref{poinmetric}) leads now to 
the definition of the (square of the) Poincar\'e length element at $q\in
\Delta_{\hh}$:
\begin{equation*}
ds^2= \frac{|d_I q|^2}{(1-|q|^2)^2}
\end{equation*}
where $q=x+yI$ and $d_I q= dx + I dy$ (for $I\in \mathbb{S}$).

The following result has its own
independent interest, as we will see later.

\begin{teo}\label{invarianza della metricadiff} All the elements of
    the group $\mathbb{M^{*}}$ of extended M\"obius transformations   of
    $\Delta_{\hh}$ leave the Poincar\'e differential metric
    (\ref{poinmetric}) invariant.
\end{teo}
\begin{proof} To begin with we recall that, by theorem \ref{teofonda},
    all the elements of $\mathbb{M}$ can be written uniquely as:
\begin{equation} \label{formafond1}
h(q)=\alpha (q - q_0)(1-\overline{q_0}q)^{-1}  \beta^{-1}
\end{equation}
where $q_0 \in \Delta_{\mathbb{H}}$ and
where
$\alpha, \beta  \in \mathbb{S}^3_{\mathbb{H}}.$ Since the (right and
left)
multiplication by elements of $\mathbb{S}^3_{\mathbb{H}}$ obviously
leaves the Poincar\'e differential metric invariant, we are left to
prove the invariance of the differential metric under the action of
the M\"obius transformations of type 
\begin{equation}\label{moebiustrans}
g(q)=(q - q_0)(1-\overline{q_0}q)^{-1}.
\end{equation}
By lemma \ref{struttura di
    G},
  the map $g$ can be decomposed as follows
   \begin{equation}\label{moebiustransdec}
g(q)=-\overline{q_{0}}^{-1}+(-q_{0}+\overline{q_{0}}^{-1})(1-\overline{q_{0}}q)^{-1}
.\end{equation}
 Since (again by lemma \ref{struttura di
    G}) all the elements of $\mathbb{M}$ are conformal, we will compute the 
    dilation coefficients of the
    differentials of the single components of $g$.  
    The dilation coefficient of the map $q \mapsto (-q_{0}+\overline{q_{0}}^{-1})q$ 
    is
    $|(-q_{0}+\overline{q_{0}}^{-1})|=\frac{(1-|q_{0}^{2}|)}{|q_{0}|}$.
    What is written in the proof of lemma \ref{struttura di
    G} yields that the dilation coefficient of
    $(1-\overline{q_{0}}q)^{-1}$ is
    $\frac{|q_{0}|}{|1-\overline{q_{0}}q|^{2}}$. Therefore the total
    dilation coefficient of $g$ is $\frac{(1-|q_{0}|^{2})}{|1-\overline{q_{0}}q|^{2}}$.
    Since $1-|g(q)|^{2}=
    \frac{(1-|q_{0}|^{2})}{|1-\overline{q_{0}}q|^{2}}$, we have proved
    the assertion for all the M\"obius transformations. To conclude
    the proof it is enough to notice that the dilation coefficient of 
    (the differential of) the
    map $q \mapsto \overline{q}$ is equal to $1$.
    \end{proof}

The map $L$ defined in \ref{realpoints} sends two (arbitrary)
points $q_1$ and $q_2$
of $\Delta_{\mathbb{H}},$ to $0$ and $t \in \mathbb{R}^+$
(respectively), which
belong to each $\Delta_I.$
Then, as in the case of the complex disc, we find that the Poincar\'e 
distance $\delta_{\Delta}$ is such that
\begin{equation} \label{infimum}
\delta_{\Delta} (q_1,q_2)=\delta_{\Delta} (0,t)=\inf\int\limits_{l} ds =
\inf\int\limits_{l} \frac{|d_Iq|}{1-|q|^2}
\end{equation}
where the infimum has been taken on all the arcs $l$ which are
piece-wise 
differentiable  and which join $0$ and $t$. Therefore we have:
\begin{prop}
    The Poincar\'e distance $\delta_{\Delta}$ of the unit disc $\Delta_{\hh}$
    is the integrated distance of the Poincar\'e differential metric
    of $\Delta_{\hh}$.
\end{prop}

\section{Poincar\'e and Kobayashi distances on the quaternionic unit
disc}
 
Let us consider the isomorphism $\mathbb{H}\cong
\mathbb{C}+\mathbb{C}j$ which leads to the identification
$\Delta_{\mathbb{H}}\cong
\Delta_{\mathbb{C}^{2}}=:\Delta$ between the open, unit disc of
$\mathbb{H}$ and the open unit ball of $\mathbb{C}^{2}$. Now that we
have given a direct, geometrical definition of
the Poincar\'e distance $\delta_{\Delta}$ of
$\Delta_{\mathbb{H}}$, the natural question arises to find a direct
proof of the fact that there exists no isometry between $\delta_{\Delta}$  and 
the Kobayashi distance of $\Delta_{\mathbb{C}^{2}}$
(see, e.g., \cite{Rud}).
To find such a proof we begin with the following

\begin{nota}\label{coincidenza} If $\delta_{\mathbb{D}}$ denotes the Poincar\'e distance 
of the open, unit disc $\mathbb{D}$  of $\mathbb{C}$, then both the Poincar\'e
distance $\delta_{\Delta}$ and the Kobayashi distance $k_\Delta$ have the
property that 
$$
\delta_{\Delta}(0, q) = k_\Delta(0, q)=\delta_{\mathbb{D}}(0, |q|)
$$
for all $q\in \Delta$.  Moreover, the
Poincar\'e differential metric and the Kobayashi differential metric
coincide with the Euclidean differential metric at the origin of the open, unit disc of $\mathbb{H}$. 
\end{nota}
\noindent With this in mind, we will prove the following technical result:

\begin{lemma}\label{tecnico}If there exists an isometry $f:\Delta
    \to \Delta$ between the
    Kobayashi distance $k_\Delta$ and the Poincar\'e distance
    $\delta_{\Delta}$, then the identity function of $\Delta$ is an
    isometry between $k_\Delta$ and 
    $\delta_{\Delta}$, and hence $k_\Delta \equiv \delta_{\Delta}$.
    
\end{lemma}
\begin{proof} If $f$ is the identity function of $\Delta$ there is nothing to prove. 
Otherwise, let $M\in \mathbb{M}$ be a quaternionic, M\"obius transformation of
    $\Delta$ such that $M(f(0))=0$. By proposition \ref{invarianza}, 
    the function $M\circ f$ is an isometry between $k_\Delta$ and 
    $\delta_{\Delta}$ which fixes $0$. If we identify $\mathbb{H}$ with $\mathbb{R}^4$, 
    then remark \ref{coincidenza} yields that the real differential  $d(M\circ
    f)_{0}$ is an orthogonal matrix. Now the geometrical definition of
    the Poincar\'e distance $\delta_{\Delta}$ given in (\ref{definizionedistanza}) 
    makes it clear that any orthogonal transformation of $\Delta$ is a 
    $\delta_{\Delta}$-isometry together with
    its inverse. Therefore the
    function $F=d(M\circ f)^{-1}_{0}\circ M \circ f : \Delta \to \Delta$
    is an isometry between $k_\Delta$ and $\delta_{\Delta}$, whose differential
    $dF_{0}$ is the identity function. Since the geodesic curves of both
    $k_\Delta$ and $\delta_{\Delta}$ passing through $0$ are the diameters of
    $\Delta$, then the isometry $F$ itself is the identity map.
\end{proof}

\noindent Given any two points $q_{1}, q_{2}\in \Delta$ there exist a
quaternionic M\"obius transformation $M$ of $\Delta_{\mathbb{H}}$ 
and a complex M\"obius transformation
$\phi$  of $\Delta_{\mathbb{C}^{2}}$ such that
$M(q_{1})=0=\phi(q_{1})$. 
Now $M$ and $\phi$ leave invariant, respectively,  $\delta_{\Delta}$ and
$k_\Delta$, and we want to investigate the relation between
$|M(q_{2})|$ and $|\phi(q_{2})|$. 
Consider $q_{1}=\alpha \cong (\alpha, 0)$ and $q_{2}=\beta
j\cong (0, \beta)$  with $\alpha, \beta \in \mathbb{C}$. 
Thanks to theorem \ref{teofonda}, 
choose $M$ to be 
$$
M(q)=(q-\alpha)(1-\overline{\alpha}q)^{-1}
$$
and with reference to \cite{Rud} choose 
$$
\phi_{(\alpha, 0)}(z,w)=\frac{(\alpha, 0) - (z, 0)
-(1-|\alpha|^{2})^{1/2}(0,w)}{1-z\overline{\alpha}}.
$$
We get 
\begin{equation}\label{Q}
|M(\beta j)|^{2}= |(\beta j - \alpha)(1-\overline{\alpha}\beta
j)^{-1}|^{2}=
\frac{(|\beta|^{2}+|\alpha|^{2})}{(1+|\alpha|^{2}|\beta|^{2})}
\end{equation}
and 
\begin{equation}\label{C}
|\phi_{(\alpha, 0)}(0,\beta)|^{2}=|(\alpha, 
-(1-|\alpha|^{2})^{1/2}\beta)|^{2}=|\alpha|^{2}+(1-|\alpha|^{2})|\beta|^{2}.
\end{equation}
Since the equality among (\ref{Q}) and (\ref{C}) does not hold in
general (due to the identity principle for real polynomials), remark \ref{coincidenza} leads to the following 

\begin{lemma} The identity map of $\Delta$ is not an isometry between the Kobayashi distance
$k_{\Delta}$ and the Poincar\'e distance $\delta_{\Delta}$. In particular $k_{\Delta}$ and 
$\delta_{\Delta}$ do not coincide.
\end{lemma}

\noindent As a direct corollary of the last two lemmas, and in accordance with a
classical consequence of the classification of non compact, rank 1, 
symmetric spaces (see, e.g., \cite{DNF},  \cite{H}), we can now state the following

\begin{teo}
There exists no isometry between the quaternionic Poincar\'e distance
and the Kobayashi distance of $\Delta_{\mathbb{H}}\cong
\Delta_{\mathbb{C}^{2}}$.
\end{teo}

Notice that, as already mentioned in the proof of theorem
\ref{distanza}, the Poincar\'e distance
and the Kobayashi distance coincide on the subsets of $\Delta$ of type
$\Delta_I=\Delta \cap L_I,$ where $L_{I} = \{ x+yI : x,y \in
\mathbb{R} \}$, for any $I \in \mathbb{S}$. Thanks to our geometrical 
approach, it is also immediate to
verify that all the real, sectional curvatures of the quaternionic, 
Poincar\'e differential metric at $0$ - and hence by homogeneity at
all points of
$\Delta$ - coincide with a same negative constant. As it is known, this is not the
case for the Kobayashi differential metric of
$\Delta_{\mathbb{C}^{2}}$ (see, e.g., \cite{KN}), for which only the
holomorphic, sectional curvatures at all points coincide with a same negative
constant.

\section{M\"obius transformations and the Poincar\'e distance on  $\hh^{+}$}

Similarly to what happens in the case of the complex plane, the
quaternionic half-space $\mathbb{H}^{+}=\{ q\in \mathbb{H} : \Re 
e(q)>0\}$ is diffeomorphic to the open, unit  disc
$\Delta_{\mathbb{H}}$ via the (biregular) Cayley 
transformation $\psi(q)=(1+q)(1-q)^{-1}\in \mathbb{G}$, (see
\cite{GS}). We can state here the following

\begin{lemma}\label{isometria} The Cayley transformation
    $\psi(q)=(1+q)(1-q)^{-1}$ maps non-Euclidean lines of
$\Delta_{\mathbb{H}}$
    onto  real, affine, half-lines or arcs of circles, which are
orthogonal to 
    $\partial\mathbb{H}^{+}$. 
\end{lemma}
\begin{proof} By corollary \ref{cfreret}, the conformality of
$\psi$ (see lemma \ref{struttura di G}), and the fact that it
 transforms $\mathbb{S}^{3}_{\mathbb{H}}$ onto
$\partial\mathbb{H}^{+}$, lead to the conclusion.
\end{proof}

It becomes now easy to define the Poincar\'e-type distance on 
$\mathbb{H}^{+}$. Given any two points $q_{1} \neq q_{2}\in
\mathbb{H}^{+}$, we can in fact consider the unique, affine, half-line
or arc of circle, $l$, of $\mathbb{H}^{+}$ which contains $q_{1},
q_{2}$ and
intersects $\partial\mathbb{H}^{+}$ orthogonally. 
We call such an $l$ the {\it non-Euclidean line} (of
$\mathbb{H}^{+}$) containing
$q_{1}$ and $q_{2}$.

We then define the two
intersections $q_{3}, q_{4}$ of $l$ with $\partial\mathbb{H}^{+}$ to
be the {\it ends} of 
$l$ (one of them might be $\infty$)  in such a
way that $q_{1}, q_{2}, q_{3}, q_{4}$ are arranged cyclically on $l$. 
Then we set 
\begin{equation}\label{poincare H} 
    \omega(q_{1}, q_{2})= \frac{1}{2}\log (\ratio(q_1,q_2,q_3,q_4))
\end{equation}
to be the {\it Poincar\'e distance}  (of $\mathbb{H}^{+}$)  between
$q_{1}$ and $q_{2}$.

\begin{teo}\label{iso} The map $w$ is a distance and the Cayley
    transformation $\psi: \Delta_{\mathbb{H}} \to \mathbb{H}^{+}$ is
an isometry 
    with respect to the Poincar\'e distances
    of $\Delta_{\mathbb{H}}$ and $\mathbb{H}^{+}$.
    \end{teo}
\begin{proof} Corollary \ref{crreale} implies that $\psi$ leaves
    invariant the
    cross ratio of four points belonging to a same non-Euclidean line 
    of $\Delta_{\mathbb{H}}$. The assertion follows.
    \end{proof}
    
We end the paper with the description of all the isometries of the
Poincar\'e distance $\omega$ of $\mathbb{H}^{+}$.

\begin{prop} The Poincar\'e distance of $\mathbb{H}^{+}$ is invariant 
    under the action of the group $\mathbb{M}^{*}(\hh^{+}):=\psi\mathbb{M}^{*}\psi^{-1}$, where
    $\mathbb{M}^{*}$ is the group of extended M\"obius transformations
and $\psi$ is the Cayley transformation.
    \end{prop}
    
\begin{definiz}\label{GLH+} The group af all linear fractional
    transformations of $\mathbb{H}^{+}$ will be denoted by $\GL$ and
    called the group of {\it M\"obius transformations} of $\hh^{+}$.
\end{definiz}

The study of an explicit description of the group 
$\GL=\psi\mathbb{M}\psi^{-1}$ of all fractional
linear transformations of $\mathbb{H}^{+}$ has a natural independent
interest. We will perform it here, starting from the identification
of 
its isotropy subgroup $\GL_{\infty}$ at the point $\infty$.

\begin{prop}\label{isotropia} If $g \in \GL_{\infty}$, then there
exist 
    $b, d\in \hh$ with $d\neq 0$ and $\Re e(bd^{-1})=0$,  such 
    that $g$ is the fractional linear transformation associated to the
    matrix
    \begin{equation}\label{forma isotropia}
\left[ \begin{array}{ll}

             |d|^{-2} d & b \\
             0 & d \\
             \end{array} \right]
	\end{equation}
    that is
\begin{equation}\label{forma isotropia 2}
g(q)= |d|^{-2}dqd^{-1}+bd^{-1}.
	\end{equation}    
\end{prop}
\begin{proof} Let $g(q)=(aq+b)(cq+d)^{-1}$. Condition
$g(\infty)=\infty$ 
    implies that $c=0$. Since $g\in \GL$, if $\Re e(q)=0$ then $\Re
    e(g(q))=0$. Therefore, if we set $q=yI$, with $y\in \rr$ and
$I\in 
    \s$, we have $\Re e(ayId^{-1}) + \Re e(bd^{-1})=0$ for all $y\in
    \rr$ and all $I\in \s$. This is equivalent to require $\Re
    e(aId^{-1})=0$ for all $I\in \s$ and $\Re e(bd^{-1})=0$.
    Let us set $a=a_{0}+a_{1}L$ and $d^{-1}=d_{0}+d_{1}M$, with
$a_{0},
    a_{1}, d_{0}, d_{1} \in \rr$ and $L, M \in \s$. The equality 
    $\Re e(aId^{-1})=0$ becomes 
    \begin{equation}\label{reale generale}
    \Re e(a_{0}d_{0}I+a_{0}d_{1}IM+a_{1}d_{0}LI+a_{1}d_{1}LIM)=0.
    \end{equation}
    Now, if $I$ is
    orthogonal to both $L$ and $M$  then (\ref{reale generale})
becomes
    (see \cite{GS} for notations)
    \begin{equation}\label{reale particolare}
    \Re e(a_{1}d_{1}LIM)=a_{1}d_{1}\langle L\times I, M\rangle=0.
    \end{equation} 
If $L$ and $M$ are $\rr$-linearly
    dependent, equation (\ref{reale particolare}) gives no conditions.
   If, otherwise,  $L$ and $M$ are $\rr$-linearly
    independent it implies $a_{1}d_{1}=0$ which directly yields
    $a, d\in \rr$, and the assertion follows. We can therefore suppose
    from now on that $L$ and $M$ are $\rr$-linearly
    dependent. If we choose $I=L=M$ equation (\ref{reale generale})
reduces 
    to
     \begin{equation}\label{reale terzocaso}
    a_{0}d_{1}+a_{1}d_{0}=0.
    \end{equation}
    Since $a_{0}d_{1}+a_{1}d_{0}=\Im m(ad^{-1})$, this leads to
    $ad^{-1}\in \rr$, i.e. $a=rd$ for some $r\in \rr$. Taking into
    occount that $\Re
    e(q)>0$ implies $\Re e(g(q))>0$, the real number $r$ has to be
   strictly positive, and can be chosen to be equal to
   $|d|^{-2}$ without loss of generality. The assertion is proved.
\end{proof}
\begin{nota} \label{conjugate}
Observe that, as a consequence of our choice of $r$, the product of
$|d|^{-2}d$ for $\overline{d},$ i.e. of one element for the conjugate
of the other on the principal diagonal of matrices in $\GL_{\infty},$
is $1.$
\end{nota}    
Let us consider the element $f_{\gamma}\in \GL$ defined by 
\begin{equation}\label{infinito finito}
    f_{\gamma}(q)=(q-\gamma)^{-1}
    \end{equation}
and associated to the matrix 
\begin{equation}\label{infinito finito matrice}
\left[ \begin{array}{ll}

             0 & 1 \\
             1 & -\gamma \\
             \end{array} \right]
\end{equation}
with $\Re e(\gamma) = 0$. We have that $f_{\gamma}(\gamma) = \infty$.

\begin{teo}\label{struttura gruppo GLH+} Let $g\in \GL \setminus
    \GL_{\infty}$. Then there
    exist $\alpha, \beta, \gamma \in \hh$  with $\alpha \neq 0$ and $\Re
    e(\gamma)=0=\Re e(\beta\alpha^{-1})$ such that
    $g(q)=(|\alpha|^{-2} \gamma\alpha q + \gamma\beta +
\alpha)(|\alpha|^{-2} \alpha q +
    \beta)^{-1}$ is associated to the matrix 
    \begin{equation}\label{matrice generale}
\left[ \begin{array}{ll}

             |\alpha|^{-2} \gamma\alpha & \gamma\beta + \alpha \\
             |\alpha|^{-2} \alpha & \beta \\
             \end{array} \right].
\end{equation}
\end{teo}
\begin{proof}
Let $g(q)=(aq+b)(cq+d)^{-1}$ be a fractional linear transformation of
$\mathbb{H}^+$ such that  $g(\infty)=ac^{-1}=\gamma\in
\partial\hh^+=\{ q\in \hh : \Re e(q)=0\}.$ Then $(f_{\gamma} \circ
g)$ fixes $\infty$ and by Proposition \ref{isotropia} there exist
$\alpha, \beta \in \mathbb{H}$ with $\Re e (\beta \alpha^{-1})=0$
such that:
\begin{equation}
\left[ \begin{array}{ll}

              0 & 1 \\
              1 & -ac^{-1} \\
             \end{array} \right]
             \cdot
\left[ \begin{array}{ll}

              a & b \\
              c & d \\
             \end{array} \right]
             =
  \left[ \begin{array}{ll}

              |\alpha|^{-2} \alpha & \beta \\
              0 & \alpha \\
             \end{array} \right].      
\end{equation}
Therefore:
\begin{equation}
\left[ \begin{array}{ll}

              a & b \\
              c & d \\
             \end{array} \right]
             =
             \left[ \begin{array}{ll}

              \gamma & 1 \\
              1 & 0 \\
             \end{array} \right]
             \cdot
             \left[ \begin{array}{ll}

              |\alpha|^{-2} \alpha & \beta \\
              0 & \alpha \\
             \end{array} \right]
             =
           \left[ \begin{array}{ll}

             |\alpha|^{-2} \gamma\alpha & \gamma\beta + \alpha \\
             |\alpha|^{-2} \alpha & \beta \\
             \end{array} \right].            
             \end{equation}
\end{proof}

The above results give us an idea on how to describe the group of all
M\"obius transformations of $\mathbb{H}^+$ in a more direct form.
In fact we will prove now that 

\begin{teo}\label{glh+}  If $K=\left[ \begin{array}{ll}

              0 & 1 \\
              1 & 0 \\
             \end{array} \right] $, 
    then the set of matrices defined by 
     \begin{equation*} 
	     \GLH = \left\{   A\in M(2,\mathbb{H}) :\ \ 
	     ^{t}\overline{A}KA=K \right\}
\end{equation*}
is a subgroup of $SL(2, \hh)$ of real dimension 10.  Moreover, 
\begin{equation*}
\GLH = \left\{   \left[ \begin{array}{ll}

              a & b \\
              c & d \\
             \end{array} \right] \in M(2, \mathbb{H}):\,\,\,  
	     \Re e (a\overline{c})=0,\,\,\, \Re e
(b\overline{d})=0, \,\,\, \overline{b}c+\overline{d}a=1  \right\}.
\end{equation*}

\end{teo} 

\begin{proof} Since
$det_{\mathbb{H}}(^{t}\overline{A})=det_{\mathbb{H}}(A)$, for any $A\in M(2, \mathbb{H})$, the relation
$^{t}\overline{A}KA=K$ implies, via the Binet property (proposition \ref{Binet})
that $\GLH \subset SL(2, \mathbb{H})$. Let us now notice that
\begin{equation}
    ^{t}\overline{AB}=^{t}\overline{B}\ ^{t}\overline{A}
\end{equation}
for all $A, B \in SL(2,\mathbb{H})$, which implies
\begin{equation}
(^{t}\overline{A})^{-1}=^{t}\overline{(A^{-1})}
\end{equation}
for all $A \in SL(2,\mathbb{H})$. We will prove that $\GLH$ is a
group. In fact, for all $A, B \in
\GLH$, we have
\begin{equation}
^{t}\overline{(AB)}K(AB)= ^{t}\overline{B}\ (^{t}\overline{A}KA)B=
^{t}\overline{B}KB=K
\end{equation}
and therefore $AB\in \GLH$. Moreover, since by definition $^{t}\overline{A}KA=K$, we obtain
\begin{equation}
^{t}\overline{(A^{-1})}KA^{-1}=(^{t}\overline{A})^{-1}KA^{-1}=K.
\end{equation}
i.e., $A^{-1}\in \GLH$ for all $A\in \GLH$. Take now any
\begin{equation*}
 A=\left[ \begin{array}{ll}

              a  & b \\
             c & d \\
             \end{array} \right]   \in M(2, \mathbb{H}).
\end{equation*}
Since 
\begin{equation*}
^{t}\overline{A}KA=\left[ \begin{array}{ll}

              2\Re e(a\overline{c})  & \overline{c}b+\overline{a}d \\
              \overline{b}c+\overline{d}a & 2\Re e(b\overline{d})\\
             \end{array} \right] 
	     \end{equation*}
we obtain that $A\in \GLH$ if and only if $\Re e (a\overline{c})=0,\,\,\, \Re e
(b\overline{d})=0, \,\,\, \overline{b}c+\overline{d}a=1$.	     
Finally, a direct computation
shows that the real dimension of $\GLH$ is $10$.
\end{proof}

The following group isomorphism assumes  an interesting geometrical meaning. Let $C=\left[ \begin{array}{rr}
             1 & 1 \\
            -1 & 1 \\
             \end{array} \right]$ be the matrix associated to the
	     Cayley transform (see lemma \ref{isometria}), and define $\Phi:
	     \GLH \to Sp(1,1)$ as the conjugation map
	     $ \Phi(A)=C^{-1}AC.$ We have

\begin{teo}\label{isomorfismo} The group $\GLH$ is isomorphic to $Sp(1,1)$
    via the application $\Phi: A \mapsto C^{-1}AC$.
    \end{teo}
    \begin{proof} Recall that $H=\left[ \begin{array}{rr}
             1 & 0 \\
             0 & -1 \\
             \end{array} \right]$ and  $K=\left[ \begin{array}{rr}
             0 & 1 \\
             1 & 0 \\
             \end{array} \right]$. It is easy to verify that
	     $C^{-1}=\frac{1}{2}^{t}C$, 
	$^{t}C^{-1}HC^{-1}=-\frac{1}{2}K$ and $^{t}CKC=-2H$.
	As a consequence, by applying theorem \ref{glh+}, we 
	obtain 
\begin{eqnarray*}
    ^{t}\overline{[\Phi(A)]}H[\Phi(A)]=\,\,^{t}\overline{(C^{-1}AC)}HC^{-1}AC=\,\,^{t}C\
    ^{t}
  \overline{A}\ (^{t}C^{-1}HC^{-1})AC\\ \nonumber
=-\frac{1}{2}\ ^{t}C\ ^{t}\overline{A}KAC=
  -\frac{1}{2}\ ^{t}CKC=2\frac{1}{2}H =H
             \end{eqnarray*}
for all $A\in \GLH$. This last equality proves the inclusion $\Phi(\GLH)\subseteq
Sp(1,1)$. Moreover
\begin{eqnarray*}
    ^{t}\overline{[\Phi^{-1}(M)]}K[\Phi^{-1}(M)]
    =\,\, ^{t}(\overline{CMC^{-1}}) K
	     (CMC^{-1}) 
	     =\,\, ^{t}C^{-1}\ 
	     ^{t}\overline{M}(^{t}C K C) MC^{-1} \\ 
	     = -2\  ^{t}C^{-1}
	     (^{t}\overline{M} H M) C^{-1} = -2\  (^{t}C^{-1} H C^{-1})= 2 \frac{1}{2} K=K
	     \end{eqnarray*}
for all $M\in Sp(1,1)$, which implies $Sp(1,1)\subseteq \Phi(\GLH)$ and ends the proof, since
$\Phi$ is obviously an injective homomorphism.	    
    \end{proof}

We will conclude the paper with the following result, 
which urges a comparison with the complex case.

\begin{teo}\label{omomorfismo finale}The map 
\begin{eqnarray}
   &\Psi: \GLH \to \GL\nonumber \\
   &A=\left[ \begin{array}{ll}
             a & b \\
             c & d \\
             \end{array} \right]\mapsto
L_{A}(q)=(aq+b)\cdot(cq+d)^{-1}
    \end{eqnarray}
    is a group homomorphism 
    whose kernel is the center of $\GLH$, that is the
    subgroup 
    $$\left\{\pm \left[ \begin{array}{ll}

             1& 0 \\
             0 & 1\\
             \end{array} \right] 
\right\}.
	     $$
\end{teo}
\begin{proof}
Since the subsets of matrices $A, N \subset GL(2, \hh)$ defined by
\begin{equation*}
  N = \left\{ \left[ \begin{array}{ll}

             |\alpha|^{-2} \gamma\alpha & \gamma\beta + \alpha \\
             |\alpha|^{-2} \alpha & \beta \\
             \end{array} \right] : \alpha, \beta, \gamma \in \hh,
	     \Re e(\gamma)=0, \Re e(\beta\overline{\alpha})=0 \right\}
\end{equation*}
\begin{equation*}
A = \left\{ \left[ \begin{array}{ll}

             |d|^{-2} d & b \\
             0 & d \\
             \end{array} \right] : b, d\in \hh, \Re e(b\overline{d})=0
	     \right\}
\end{equation*}
are contained in $\GLH$ then, by proposition \ref{forma isotropia} and theorem 
\ref{struttura gruppo GLH+}, $\GL\subseteq \Psi(\GLH)$. Now, if $A= \left[ \begin{array}{ll}
              a & b \\
              c & d \\
             \end{array} \right] \in \GLH$, then 
	     $\overline{d}a+\overline{b}c=1$,  $\Re e
	     (a\overline{c})=0$
	     and $\Re e(b\overline{d})=0$. If
	     $q\in\hh\setminus\rr$ we have
	     \begin{eqnarray*}
		& \Re e((aq+b)(cq+d)^{-1})=\Re
		 e((aq+b)\overline{(cq+d)})\\
		&= \Re
		 e((aq+b)(\overline{q}\ \overline{c}+\overline{d}))=
		 \Re
		 e(a|q|^{2}\overline{c}+aq\overline{d}+b\overline{q}\
		 \overline{c}+b\overline{d})\\
		 &=\Re
		 e(aq\overline{d}+b\overline{q}\
		 \overline{c})=\Re
		 e(q\overline{d}a-q\overline{c}b)=\Re
		 e(q(\overline{d}a-\overline{c}b))\\
		 &=\Re
		 e(q(1-(\overline{b}c+\overline{c}b))=0.
\end{eqnarray*}	     
The last equalities lead to  
$\Psi(A)(\partial\hh^+)=\partial\hh^+$. Since $\Re e(\Psi(A)(1))=1$,  
we have 
$\Psi(\GLH)\subseteq \GL$. The same argument used in theorem \ref{homo} leads to 
the identification of the kernel of $\Psi$ and allows at this point the conclusion of the proof.
\end{proof}
 
The new description of the group of quaternionic, M\"obius transformations
$\mathbb{M}(\hh^{+})$ of $\hh^{+}$ in terms of the group of matrices
$SL(\hh^{+})$
is interesting and promises
developments in several directions. We plan to investigate, for
example, the analogous of the Fuchsian subgroups and the possibility
of constructing Riemann-type $\hh$-surfaces.
\vskip 0.2cm
The Cayley
transformation $\psi: \Delta_{\hh} \to \hh^{+}$, defined by 
$\psi(q)=(1+q)(1-q)^{-1}$, has real coefficients,  and therefore it maps
every $L_{I}=\rr+I\rr \cong \mathbb{C} \ \ (I\in\s)$ onto itself. The
argument used in the complex case 
leads to the definition of the
Poincar\'e differential metric on $\hh^{+}$: for any $q \in
\hh^{+}$ and any $\tau\in \hh$ the length of the vector $\tau$ for the
Poincar\'e metric  at $q$ is expressed by
\begin{equation}\label{poincareaccapiu}
    \langle \tau \rangle_{q}=\frac{|\tau|}{2|\Re e(q)|}.
\end{equation}
Formula (\ref{poincareaccapiu}) leads as before to the definition of
the (square of the) Poincar\'e length element in $\hh^{+}$:
\begin{equation*}
ds_{\hh^{+}}^{2}=\frac{|d_{I}q|^{2}}{4|\Re e(q)|^{2}}.
\end{equation*}

We end the paper by stating the following result, whose proof is
straightforward:
\begin{teo}
    The Poincar\'e differential metric (\ref{poincareaccapiu}) is invariant under
    the action of the group $\mathbb{M}^{*}(\hh^{+})$ of all extended 
    M\"obius transformations of $\hh^{+}$. Moreover, the Poincar\'e distance $\omega$ of
    $\hh^{+}$ defined by  $\omega(q_{1}, q_{2})= \frac{1}{2}\log
    (\ratio(q_1,q_2,q_3,q_4))$ is the integrated distance of the
    Poincar\'e differential metric.
\end{teo}

\bibliographystyle{amsplain}

\end{document}